\documentclass[11pt]{amsart}
\usepackage{amsfonts}
\usepackage{amsmath}
\usepackage{amsthm}
\usepackage{mathrsfs}
\usepackage{bm,bbm}
\usepackage[dvips]{graphicx}
\usepackage{caption}
\usepackage{color}

\usepackage{bigints}

\numberwithin{equation}{section}

\allowdisplaybreaks

\theoremstyle{plain}
\newtheorem{theorem}{Theorem}[section]
\newtheorem{lemma}[theorem]{Lemma}
\newtheorem{proposition}[theorem]{Proposition}

\theoremstyle{definition}
\newtheorem{definition}[theorem]{Definition}
\newtheorem{remark}[theorem]{Remark}

\def\beqn{\begin{equation}}
\def\beqn*{$$}
\def\eeqn{\end{equation}}

\def\P{\mathbb{P}}
\def\E{\mathbb{E}}

\newcommand{\reals}{{\mathbb R}}

\newcommand{\R}{\reals}
\newcommand{\bbn}{{\mathbb N}}

\newcommand{\vep}{\varepsilon}

\newcommand{\one}{{\mathbbm 1}}

\newcommand{\remove}[1]{}

\newcommand{\A}{\mathcal A}

\newcommand{\bp}{{\bf p}}

\newcommand{\al}{\alpha}

\newcommand{\U}{\mathcal U}

\newcommand{\taual}{{\tau_{m+1} + \alpha_{m+1}}}
\newcommand{\mH}{\mathcal H}

\newcommand{\mL}{\mathcal L}
\newcommand{\mT}{\mathcal T}

\newcommand{\ER}{Erd\H{o}s-R\'enyi }

\begin{document}

\bibliographystyle{abbrv}

\renewcommand{\baselinestretch}{1.05}

\title[Limit theorems for Betti numbers]
{Limit theorems for high-dimensional Betti numbers in the multiparameter random simplicial complexes}

\author{Takashi Owada}
\address{Department of Statistics\\
Purdue University \\
IN, 47907, USA}
\email{owada@purdue.edu}

\author{Gennady Samorodnitsky}
\address{School of Operations Research and Information Engineering\\
Cornell University \\
NY, 14853, USA}
\email{gs18@cornell.edu}

\thanks{Owada's research was partially supported by the AFOSR grant
  FA9550-22-1-0238 at Purdue University. Samorodnitsky's research was partially supported by the AFOSR grant
  FA9550-22-1-0091 at Cornell University.}

\keywords{strong law of large numbers, central limit theorem, large deviation, Betti number, multiparameter random simplicial complex. \vspace{.5ex}}

\begin{abstract}
We consider the multiparameter random simplicial complex on a vertex
set $\{ 1,\dots,n \}$, which is parameterized by multiple connectivity
probabilities. Our key results concern the topology of this complex of
dimensions higher than the critical dimension. We show that the
higher-dimensional Betti numbers satisfy strong laws of large numbers
and central limit theorems. Moreover, lower tail large deviations for
these Betti numbers are also discussed. Some of our results indicate
an occurrence of phase transitions in terms of the scaling constants
of the central limit theorem, and  the exponentially decaying rate of
convergence of lower tail large  deviation probabilities. 
\end{abstract}

\maketitle

\section{Introduction}

The main theme of this paper is topological features of large
high-dimensional complex systems. A  large complex system of interest
to us is a random simplicial complex, in which minor changes in a
low-dimensional and local  structure may significantly affect
high-dimensional and global structure of  the entire system. The
random simplicial complex considered in this paper is a
higher-dimensional analog of the \ER random graph on $[n]:=\{
1,\dots,n \}$, where edges are added with probability $p$
independently of other potential edges. The literature contains
several higher-dimensional extensions of the \ER random graph. One is
the \emph{random clique  complex}, in which a set of vertices in $[n]$
forms a face (or a simplex) if and only if these vertices form a
clique in the \ER random graph (\cite{kahle:2009,
  thoppe:yogeshwaran:adler:2016}). The \emph{Linial--Meshulam complex} \cite{linial:meshulam:2006}
is a $2$-dimensional random simplicial complex, which starts with the
full $1$-dimensional skeleton of the complete simplicial complex and
adds each $2$-dimensional face, independently with probability
$p$. This complex was  extended to its $d$-dimensional version by
Meshulam and Wallach (\cite{meshulam:wallach:2009}). In this paper we
study the \emph{multiparameter random simplicial complex} introduced
by Costa and Farber (\cite{costa:farber:2016, costa:farber:2017}) as a
generalization of all the models described above. This is an abstract
simplicial complex on $[n]$, which is parameterized by a sequence of
connectivity probabilities $\bp = (p_1,p_2,\dots,p_{n-1}) \in
[0,1]^{n-1}$. One can construct this complex  recursively. First we
take $[n]$ as the $0$-dimensional skeleton and add edges between pairs
of vertices independently with probability $p_1$, just like in the
\ER random graph. For each $i=2,\dots,n-1$, an $i$-dimensional face
(henceforth we simply say ``$i$-face") is added with probability
$p_i$, independently of all other potential faces, conditioned on the
presence of all the $(i-1)$-faces  that constitute a boundary of that
$i$-face. The complex constructed  this way is denoted as $X([n],\bp)$. 

A number of studies of the multiparameter random simplicial complexes
have appeared recently, many of which focus on the asymptotic, as
$n\to\infty$, behavior of  topological invariants of the complex, such
as the \emph{Betti numbers} and the  \emph{Euler characteristic}. The
model typically contains a particular dimension in which the Betti
number  dominates the Betti numbers in other dimensions. This
dimension is  called the \emph{critical dimension}, and much of the
effort has been devoted to  the analysis of Betti numbers at the
critical dimension \cite{kahle:meckes:2013, thoppe:yogeshwaran:adler:2016,
  fowler:2019, owada:samorodnitsky:thoppe:2021, kanazawa:2022}. There
are also results on the Betti numbers in dimensions lower than the
critical dimension. For example, by  employing the cohomology
vanishing theorem \cite{ballmann:swiatkowski:1997} and the spectral
gap theorem \cite{hoffman:kahle:paquette:2021}, Fowler  proved that
the cohomology groups in dimensions lower than the critical dimension
vanish with high probability  (\cite{fowler:2019}). However, very
little has been done about the asymptotics, as $n\to\infty$,  of Betti
numbers in dimension higher  than the critical dimension, except for
certain special cases in the context of random clique  complexes; see
\cite{kahle:meckes:2013,  thoppe:yogeshwaran:adler:2016}. 

In this paper we derive various limit theorems for the Betti numbers of
dimension greater than the critical dimension.   We start by
establishing the strong law of large numbers (SLLN) and the central
limit theorem (CLT) for these high-dimensional  Betti numbers. The
obtained limit theorems characterize the  ``average" behavior and the
``likely deviations" from the average behavior  of these Betti
numbers.  We also derive a lower bound for the lower tail large deviation
probabilities for the Betti numbers of
dimension greater than the critical dimension.   

This paper is structured as follows. Section
\ref{sec:setup.main.results} provides a precise setup of our
multiparameter random simplicial complex. In this section we also
present the  SLLN, the  CLT, and large deviations results on Betti
numbers in dimensions greater than the critical dimension. Some of
our results indicate presence of phase  transitions  in terms of the
scaling constants in the CLT and the exponentially decaying rate of
convergence for the lower tail large deviation probabilities. The
proofs are in Section \ref{sec:proofs}, with some auxiliary results
deferred to the Appendix. 

We conclude this Introduction with several comments on our setup and
our results. We treat only the ``static" version of the multiparameter
simplicial complex, where the complex is sampled once and for
all. The questions we are pursuing in this paper make sense also for
the dynamic multiparameter simplicial  complexes considered in our previous work \cite{owada:samorodnitsky:thoppe:2021}  and in
\cite{thoppe:yogeshwaran:adler:2016, skraba:thoppe:yogeshwaran:2020,
  fraiman:mukherjee:thoppe:2023}. In particular, the models in \cite{thoppe:yogeshwaran:adler:2016, owada:samorodnitsky:thoppe:2021} involve on/off
processes to describe the state of various ingredients, so that the complex,
though distributionally stationary, may change with time. Since even
the ``static'' case is technically involved enough, we have decided
not to pursuit this direction in the present work. Further, in our
large  deviation results, we present only  ``lower  bound" for the lower
tail probabilities for the Betti numbers. We conjecture that the
logarithmic order of magnitude of our lower bounds in  Theorem
\ref{t:lower.LD} is the tightest possible, but we do not yet have
corresponding upper bounds to justify that conjecture. 
There are,  clearly, a number of other 
follow-up questions on large deviations. These questions are postponed to future work. 

Throughout the paper, we will use the following notation. The
cardinality of a set $A$ is denoted by $|A|$.  
 Given a sequence $(a_i)_{i=1}^\infty$ of real numbers, we denote
 $\bigvee_{i=1}^\infty a_i:=\sup_{i\ge1} a_i$ and
 $\bigwedge_{i=1}^\infty a_i := \inf_{i\ge1}a_i$.  Moreover,
 $\Rightarrow$  denotes weak convergence in $\R$, $\mathcal
 N(a,b^2)$ is a normal distribution with mean $a$ and variance $b^2$, 
and $\one \{  \cdot\}$ represents an indicator function. Finally, we
will use the capital letter $C$ and its versions (such as $C_j, \,
C^*$, etc) to denote finite positive constants whose exact values are 
not important and  can change from one appearance to the next. 

\section{Setup and main results} \label{sec:setup.main.results}

We start with a precise setup for the multiparameter random simplicial
complex $X([n],\bp)$ described in the previous section. 
As in the previous studies  \cite{costa:farber:2016,
  costa:farber:2017, fowler:2019, owada:samorodnitsky:thoppe:2021},
we choose the connectivity probabilities $\bp =
(p_1,p_2,\dots,p_{n-1})$  to be the following  functions of $n$: 
$$
p_i=n^{-\al_i}, \ \ i=1,2,\dots,n-1, 
$$
where each $\al_i\in [0,\infty]$ is a parameter. Given a sequence of parameters $(\al_1,\al_2,\dots,\al_{n-1})$, let 
\begin{equation}  \label{e:def.psi}
\psi_j := \sum_{i=1}^j \binom{j}{i} \al_i, \ \ j\ge1. 
\end{equation}
Define $q:= \min\{ i\ge1: \al_i>0 \}$, the smallest dimension $i$ such that $p_i$ goes to $0$ as $n\to\infty$. 
Notice that \eqref{e:def.psi} is increasing in the sense
of $\psi_i <  \psi_j$ for all $q\leq i< j$. We assume, as is common  in
the literature (\cite{costa:farber:2016, costa:farber:2017,
  fowler:2019, owada:samorodnitsky:thoppe:2021}), that  the system has
a \emph{critical}  dimension $k$, which is a dimension satisfying 
$$
0 < \psi_q < \dots <\psi_k < 1 < \psi_{k+1}< \dots, \ \ \text{with } \ k \ge q. 
$$
Next,  define 
$$
\tau_j := j+1 -\sum_{i=1}^j \psi_i = j+1 -\sum_{i=1}^j \binom{j+1}{i+1} \al_i, \ \ 0 \le j \le n-1
$$
(by convention we set $\sum_{i=1}^0\equiv 0$). Then, the monotonicity of $(\psi_j)$ ensures that 
\begin{equation}  \label{e:tau.maximized.at.crit}
q = \tau_{q-1} < \tau_q < \cdots < \tau_k > \tau_{k+1} > \cdots,
\end{equation}
 which also implies that  $(\tau_j)$ is maximized at the critical dimension $k$. Furthermore, $\tau_j$ can be negative for  large $j$.  

The main theme of this paper is the asymptotic behavior of the Betti numbers of dimension higher than the critical dimension, i.e.,  $\beta_{m,n}:= \beta_m \big( X([n], \bp) \big)$ for $m\ge k+1$. 
We obtain the SLLN and the CLT for $(\beta_{m,n})_{n\ge1}$. Moreover,
a lower bound for the lower tail large deviations for $(\beta_{m,n})_{n\ge1}$ will be derived as well. 
To state our main results,  fix $m\ge k+1$. We prove the the SLLN and
the CLT under the assumption 
\begin{equation}  \label{e:pos.assump}
\taual >0,  \ \  \text{ and } \  \ \al_{m+1}   \in (0,\infty). 
\end{equation}
Our first result  is the SLLN for $(\beta_{m,n})_{n\ge1}$.

\begin{theorem}     \label{t:SLLN}
Under the assumption \eqref{e:pos.assump}, we have, as $n\to\infty$, 
$$
\frac{\beta_{m,n}}{n^{\tau_{m+1}+\al_{m+1}}} \to \frac{1}{(m+2)!},  \ \ \text{a.s.}
$$
\end{theorem}
\begin{remark} \label{rk:SLLN.rk}
For comparison, the  SLLN for the Betti numbers at the critical
dimension $k$ was established in
\cite{owada:samorodnitsky:thoppe:2021}:  as $n\to\infty$,  
\begin{equation}  \label{e:SLLN}
\frac{\beta_{k,n}}{n^{\tau_k}} \to \frac{1}{(k+1)!}, \ \ \text{a.s.}
\end{equation}
The proof of \eqref{e:SLLN} heavily depends on the Morse inequality, 
$$
f_{k,n}-f_{k+1,n}-f_{k-1,n} \le \beta_{k,n} \le f_{k,n},  
$$
where $f_{j,n}$ denotes the simplex count of dimension $j$ in $X([n],
\bp)$. It was shown in \cite{owada:samorodnitsky:thoppe:2021}  that
$f_{k,n}$ dominates $f_{k\pm 1, n}$ as $n\to\infty$, at least in the
sense of expectations, and the SLLN for $(\beta_{k,n})_{n\ge1}$ was
obtained as a direct result of  the SLLN for $(f_{k,n})_{n\ge1}$. A
much more delicate argument is required in 
the current study,  and it is based on  an explicit representation of
the higher-dimensional  Betti numbers.  
\end{remark}
\medskip

Next, we state the CLT for $(\beta_{m,n})_{n\ge1}$ with $m\ge k+1$. It
requires a condition in addition to \eqref{e:pos.assump}. 
For $n\geq m+2$ and $A_1, A_2, A_3\subset [n]$ with $|A_i|=m+2$, set 
\begin{equation}  \label{e:def.lij.l123}
\ell_{ij} = |A_i \cap A_j|, \  \ 1\le i < j \le 3,  \ \ \ell_{123} = |A_1 \cap A_2 \cap A_3|
\end{equation}
and denote 
$$
\A_{m,k} := \big\{ (A_1,A_2,A_3)\subset [n]^3: |A_i|=m+2, \, i=1,2,3, \ \ell_{ij}\in \{ k+2,\dots,m+1  \}, \, 1\le i <j \le 3  \big\}. 
$$
We assume that 
\begin{equation}\label{e:CLT.cond}
\frac{3}{2}\big(\tau_q\wedge (\tau_{m+1}+\al_{m+1})\big) +\bigvee_{(A_1,A_2,A_3)\in \A_{m,k}} \Big\{ \tau_{\ell_{123}-1}-\tau_{\ell_{12}-1}-\tau_{\ell_{13}-1}-\tau_{\ell_{23}-1}   \Big\}<0. 
\end{equation}
Note that it is possible to avoid including a selection of sets from
$[n]$ as a part of the assumption. It is not difficult to check that
numbers $\ell_{12},\ell_{13},\ell_{23},\ell_{123}$ in
  \eqref{e:def.lij.l123} are feasible in
\eqref{e:CLT.cond} if and only if the following conditions hold:
\begin{align}
&\ell_{ij}\in \{ k+2,\dots,m+1  \}, \, 1\le i <j \le 3, \label{e:explcit.cond}\\
\notag &\max\bigl\{ \ell_{12}+\ell_{13},  \ell_{12}+\ell_{23},
  \ell_{13}+\ell_{23}\bigr\}-(m+2)\leq \ell_{123}\leq \min\bigl\{\ell_{12},\ell_{13},\ell_{23}\bigr\}.
\end{align}
Instead of referring to the restrictions \eqref{e:explcit.cond}, we
keep \eqref{e:CLT.cond} as a ``more operational'' formulation in the
proof.

\begin{theorem}   \label{t:CLT}
Assume conditions \eqref{e:pos.assump} and \eqref{e:CLT.cond}. \\
$(i)$ If $\taual >\tau_q$, then, 
$$
\frac{\beta_{m,n}-\E[\beta_{m,n}]}{n^{\taual -\tau_q/2}} \Rightarrow \mathcal N \big( 0,(q+1)! ((m-q+1)!)^2 \big), \ \ \text{as } n\to\infty. 
$$ 
$(ii)$  If $\taual <\tau_q$, then, 
$$
\frac{\beta_{m,n}-\E[\beta_{m,n}]}{n^{\frac{1}{2}(\taual)}} \Rightarrow \mathcal N ( 0,(m+2)! ), \ \ \text{as } n\to\infty. 
$$
$(iii)$ If $\taual =\tau_q$, then, 
$$
\frac{\beta_{m,n}-\E[\beta_{m,n}]}{n^{\frac{1}{2}(\taual)}} \Rightarrow \mathcal N \left( 0, \Big(\frac{1}{(q+1)! ((m-q+1)!)^2} + \frac{1}{(m+2)!}  \Big)^{-1} \right), \ \ \text{as } n\to\infty. 
$$
\end{theorem}
\begin{remark}
It is not difficult to verify that in some cases the conditions
\eqref{e:pos.assump} and \eqref{e:CLT.cond} hold at the same
time. For example, if $m=k+1$, then  \eqref{e:CLT.cond} always holds!
Indeed, we know by \eqref{e:explcit.cond} that, in this case,  
$\ell_{12}=\ell_{13}=\ell_{23}=k+2$ and $\ell_{123}\in
\{k+1,k+2\}$. Since $\tau_{k+1}< \tau_k$, one can upper
  bound \eqref{e:CLT.cond} by 
\begin{align*}
&\frac32 (\tau_{k+2}+\al_{k+2} )+\tau_k-3\tau_{k+1}
=\frac32 (\tau_{k}+2-\psi_{k+1}-\psi_{k+2}+\al_{k+2} )+\tau_k-3(\tau_{k}+1-\psi_{k+1})\\
=&-\frac12 \tau_k -\frac32 \left[  \psi_{k+2} -\alpha_{k+2}-\psi_{k+1}\right]
= -\frac12 \tau_k -\frac32  \sum_{i=1}^{k+1} \left( {k+2\choose
   i} - {k+1\choose i}\right)\alpha_i<0.
\end{align*}

On the other hand, the assumption \eqref{e:CLT.cond} is not
vacuous. A simple example of a situation where \eqref{e:CLT.cond} does
not hold can be constructed by taking 
$\ell_{12}=\ell_{13}=\ell_{23}=m+1$, $\ell_{123}=m$ for  a
  large even $m$, and choosing a
large $\alpha_{m/2}$. 
\end{remark}

It is shown in \cite{owada:samorodnitsky:thoppe:2021} that the Betti number in the critical dimension satisfies 
$$
\frac{\beta_{k,n}-\E[\beta_{k,n}]}{n^{\tau_k -\tau_q/2}} \Rightarrow
\mathcal N \big( 0,(q+1)! ((k-q)!)^2 \big), \ \ \text{as } n\to\infty, 
$$
so the power of $n$ in  the scaling depends directly on the value of
$\tau_k$. In contrast, in Theorem \ref{t:CLT}  the power of $n$ in the
scaling depends on  $\tau_{m+1}+\alpha_{m+1}$. Furthermore,  Theorem
\ref{t:CLT}  describes a phase transition not present in the critical
dimension. 
\medskip

Our final result takes a step towards understanding the lower tail
large deviations  for $(\beta_{m,n})_{n\ge1}$ with $m\ge k+1$.  
In contrast to other limit theorems, such as the SLLN and CLT
(\cite{kanazawa:2022, owada:samorodnitsky:thoppe:2021,
  thoppe:yogeshwaran:adler:2016, kahle:meckes:2013}), very little is
known about large deviations for  Betti numbers in the multiparameter
simplicial complex or versions of that model. 
The only exception is \cite{samorodnitsky:owada:2023}, which
investigates the upper tail large deviations for the Betti number in the critical dimension. 
In the result below, we are able to derive a lower bound for lower tail large
deviation probabilities for $(\beta_{m,n})_{n\ge1}$. 

\begin{theorem}  \label{t:lower.LD}
Assume condition \eqref{e:pos.assump}. Let $\vep\in (0,1)$. \\
$(i)$ If $\taual \ge \tau_q$,
$$
\liminf_{n\to \infty} \frac{1}{n^{\tau_q}} \log \P \Big( \beta_{m,n}
\le (1-\vep)\, \frac{n^\taual}{(m+2)!} \Big) \ge  -C_{m,q}\vep^2, 
$$
where 
\begin{equation}\label{e:def.C.epsilon}
C_{m,q} = \frac{5^4\cdot 2^{2(m+6)-q} \vee 2^{2(m+5)}}{(q+1)!}. 
\end{equation}
$(ii)$ If $\taual <\tau_q$, 
$$
\liminf_{n\to \infty} \frac{1}{n^{\taual}} \log \P \Big( \beta_{m,n} \le (1-\vep)\, \frac{n^\taual}{(m+2)!} \Big) \ge -\frac{5000}{(m+2)!}\, \vep^2. 
$$
\end{theorem}
\medskip

\section{Proofs}   \label{sec:proofs}
We prove Theorem \ref{t:lower.LD} first, as its proof requires
developing certain new techniques that are best seen first. The
proofs of Theorems \ref{t:SLLN} and \ref{t:CLT} are postponed to later
sections.

\subsection{Preliminaries}

Our proof  relies on an explicit representation of  Betti
numbers. Recall  that an $\ell$-dimensional complex $X$ is said
to be \emph{pure} if every face of $X$ is contained in an $\ell$-face.  

\begin{definition}  \label{def.strong.conn.type1}
Given a positive integer $\ell$, a simplicial complex $K$ is called
\emph{strongly connected} of order $\ell$ (we abbreviate it as ``$\ell$-strongly connected") if 
\vspace{3pt}

\noindent \textbf{$(i)$} every pair of $\ell$-faces $\sigma, \tau \in K$ can be connected by a sequence of $\ell$-faces 
$$
\sigma = \sigma_0, \sigma_1, \dots, \sigma_{j-1}, \sigma_j = \tau, 
$$
such that dim$(\sigma_i \cap \sigma_{i+1})=\ell-1$ for each $0\le i \le j-1$. \\
\noindent \textbf{$(ii)$} the $\ell$-skeleton of $K$ is pure. 
\end{definition}

Note that the dimension of $K$ itself can be greater than
$\ell$. Given a simplicial complex $X$, let $K$ be an $\ell$-strongly
connected subcomplex of $X$. We say that $K$ is \emph{maximal} if
there is no $\ell$-strongly connected subcomplex $K'\subset X$
containing $K$ as a strict subset. 

\begin{proposition}[Proposition A.6 in \cite{owada:samorodnitsky:thoppe:2021}]   \label{p:representation.Betti}
For every $m\ge1$, 
\begin{equation}  \label{e:representation.Betti}
\beta_{m,n} = \beta_m \big( X([n], \bp) \big) = \sum_{j=m+2}^n \sum_{r\ge1}\sum_{\sigma \subset [n], \, |\sigma|=j} r \eta_\sigma^{(j,r,m)}, 
\end{equation}
where $\eta_\sigma^{(j,r,m)}$ is the indicator function of the event
that $X(\sigma, \bp)$, representing the restriction of $X([n], \bp)$ to $\sigma$,
forms a maximal $m$-strongly connected subcomplex in 
$X([n], \bp)$ such that $\beta_m \big( X(\sigma, \bp) \big)=r$.  
\end{proposition}

We next provide an estimate of the probability that the restriction of
$X([n], \bp)$ to a finite number of points in $[n]$ contains  an $m$-strongly connected complex, 
which has at least one trivial or non-trivial $m$-cycle.  Since the
proof is  almost identical  to Lemma 7.1 of
\cite{owada:samorodnitsky:thoppe:2021}, we omit it. 

\begin{lemma} \label{l:estimate.st.conn}
Let $K$ be an $m$-strongly connected simplicial complex on $j\ge m+2$
points, which contains at least one trivial or non-trivial $m$-cycle.
Then $K$ contains at least
$$
\binom{m+2}{i+1} + (j-m-2)\binom{m}{i}
$$
faces of dimension $i=q,\ldots, m$ and, hence,  for $\sigma
\subset [n]$ with $|\sigma|=j$,  
$$
\P \big(  X(\sigma, \bp)  \ \text{contains a subcomplex
  isomorphic to}\  K \big) \le j! \, \prod_{i=q}^m p_i^{\binom{m+2}{i+1}} \Big(  \prod_{i=q}^m p_i^{\binom{m}{i}} \Big)^{j-m-2}. 
$$
Here,  the isomorphism is that of simplicial complexes.  
\end{lemma}

\subsection{Concepts and representations}    \label{sec:concepts}

The main purpose of this section is to introduce a number of concepts
together with  their short-hand notations, that are used throughout
the proof of main theorems. First, recalling the indicator function in
\eqref{e:representation.Betti}, define, for $j\ge m+2$ and $r\ge1$,  
\begin{equation}  \label{e:def.T.jr}
T_{j,r} := \sum_{\sigma\subset [n], \, |\sigma|=j} \eta_\sigma^{(j,r,m)}. 
\end{equation}
If $j=m+2$, then  $\eta_\sigma^{(m+2,r,m)}$ is   identically zero for
all $r\ge2$. In other words, $\eta_\sigma^{(m+2,1,m)}$ is the only
indicator which can be non-zero. Because of that we will drop  the
subscript ``$1$" from $T_{m+2,1}$ and simply  
write $T_{m+2}:= T_{m+2,1}$. In particular, a non-zero value of
$\eta_\sigma^{(m+2,1,m)}$ requires that $X(\sigma,\bp)$ forms a maximal
non-trivial $m$-cycle, such that $\beta_m\big( X(\sigma,\bp)
\big)=1$. By Proposition \ref{p:representation.Betti}, $\beta_{m,n}$
can be represented as  
\begin{equation}  \label{e:repre.Betti}
\beta_{m,n} = \sum_{j=m+2}^n \sum_{r\ge1} r T_{j,r} = T_{m+2} + \sum_{j=m+3}^n \sum_{r\ge1} r T_{j,r}. 
\end{equation}

Next, let us define for $j\ge m+2$, 
\begin{equation}  \label{e:def.Sj}
S_j := \sum_{\sigma\subset [n], \, |\sigma|=j}\xi_\sigma^{(j,m)}, 
\end{equation}
where $\xi_\sigma^{(j,m)}$ is the indicator function   of the event on
 which $X(\sigma, \bp)$   
contains an $m$-strongly connected subcomplex with at
least one trivial or non-trivial $m$-cycle.
If $j=m+2$, a non-zero value of the indicator
$\xi_\sigma^{(m+2,m)}$ requires that $X(\sigma, \bp)$ forms a trivial or
non-trivial $m$-cycle. When $X(\sigma, \bp)$ forms a non-trivial $m$-cycle, we
have that $\beta_m\big( X(\sigma, \bp) \big)=1$. Note also that a
``trivial $m$-cycle" means a simplicial complex of a single 
$(m+1)$-simplex.   
Similarly, for $j\ge m+2$, let 
\begin{equation}  \label{e:def.Rj}
R_j := \sum_{\sigma\subset [n], \, |\sigma|=j}\widetilde \xi_\sigma^{(j,m)},
\end{equation}
where $\widetilde \xi_\sigma^{(j,m)}$ is the indicator function 
of the event on
 which  $X(\sigma, \bp)$   
contains an $m$-strongly connected subcomplex. Clearly, 
for every $j\ge m+2$ we have 
\begin{equation}  \label{e:T.S.R.inequ}
\sum_{r\ge 1}T_{j,r}\le S_j 
\le R_j.
\end{equation}

For the proof of Theorem \ref{t:lower.LD}, one needs to reformulate
the probability space on which \eqref{e:repre.Betti},
  \eqref{e:def.Sj}, and \eqref{e:def.Rj} are defined, in the following way. Let $K_n$ be the
$(m+1)$-skeleton of a complete simplicial complex on $[n]$. In other
words, $K_n$ is the set of all simplices in $[n]$ with dimension $m+1$
or less. 
Given  connectivity probabilities in $\bp$, let $\Gamma_\bp$ be a
random subset of $K_n$ such that for every $i=1,\dots,m+1$, a word of
length $i+1$ is included in $\Gamma_\bp$ with probability
$p_i$. More precisely, $\Gamma_\bp=\Gamma_\bp(\omega)$ is  a random subset
of $K_n$ defined on the product space $\{ 0,1
\}^{\sum_{\ell=2}^{m+2}\binom{n}{\ell}}$ of Bernoulli random variables
with parameters  
$$
(\underbrace{p_1,\dots,p_1}_{\binom{n}{2}}, \underbrace{p_2,\dots,p_2}_{\binom{n}{3}}, \dots, \underbrace{p_{m+1}, \dots, p_{m+1}}_{\binom{n}{m+2}})\in [0,1]^{\sum_{\ell=2}^{m+2}\binom{n}{\ell}}. 
$$
Throughout the paper we may and will assume that the $(m+1)$-skeleton of $ X([n], \bp)$
is defined on this probability space.

\subsection{Proof of Theorem \ref{t:lower.LD}}

We begin with the following proposition. 

\begin{proposition} \label{p:lower.bdd.with.leading.term} 
Under the assumption \eqref{e:pos.assump}, for every $\eta\in \big(0,(1-\vep)/3\big)$, we have 
\begin{align}
\begin{split}  \label{e:first.leading.term}
&\liminf_{n\to\infty} \frac{1}{n^{\tau_q \wedge (\taual)}} \log \P \Big( \beta_{m,n}\le (1-\vep) \,\frac{n^\taual}{(m+2)!}  \Big)  \\
&\ge \liminf_{n\to\infty} \frac{1}{n^{\tau_q \wedge (\taual)}} \log \P \Big( S_{m+2}\le (1-\vep -2\eta) \, \frac{n^\taual}{(m+2)!}  \Big). 
\end{split}
\end{align}
\end{proposition}

\begin{proof}
Fix a positive integer $D$ such that 
\begin{equation}  \label{e:def.D}
D > \frac{m+2+\tau_m}{\psi_m-1}; 
\end{equation}
note that the fraction in the right hand side of \eqref{e:def.D} is
positive due to the first assumption in \eqref{e:pos.assump}. 
We decompose the representation \eqref{e:repre.Betti} of 
$\beta_{m,n}$ as 
\begin{equation}  \label{e:decomp.Betti}
\beta_{m,n} = T_{m+2} + \sum_{j=m+3}^{D+m} \sum_{r\ge 1} rT_{j,r} + \sum_{j=D+m+1}^n \sum_{r\ge 1} rT_{j,r}. 
\end{equation}
We  derive an upper bound for each of the  terms in
\eqref{e:decomp.Betti}. First, it is clear that
$T_{m+2}\le S_{m+2}$.  Recall next that the
$m$-dimensional Betti number of any simplicial complex on $j$ vertexes
is upper 
bounded by the corresponding $m$-face counts, which  is further upper 
bounded by $\binom{j}{m+1}$. This, together with
  \eqref{e:T.S.R.inequ}, implies that  
$$
\sum_{j=m+3}^{D+m} \sum_{r\ge 1} r T_{j,r} \le \sum_{j=m+3}^{D+m} \binom{j}{m+1}S_j. 
$$
For the remaining terms  in \eqref{e:decomp.Betti}, we use the bound   
\begin{align*}
\sum_{j=D+m+1}^n \sum_{r\ge 1} rT_{j,r} &\le n^{m+1}\sum_{j=D+m+1}^n  \sum_{\sigma \subset [n], \, |\sigma|=j} \one \big\{ X(\sigma,\bp) \text{ is a maximal } \\
&\qquad \qquad \qquad\qquad \qquad\qquad m\text{-strongly connected   subcomplex
           of} \  X([n],\bp)                                                                                                                                                          \big\} \\
&\le n^{m+1}R_{D+m+1}, 
\end{align*}
where $R_{D+m+1}$ is given in \eqref{e:def.Rj}. The last inequality
above follows from the fact that a maximal $m$-strongly connected
subcomplex of size  $j\ge D+m+1$  always contains a further
$m$-strongly connected subcomplex  on $D+m+1$ vertices, and the 
maximality prevents such an $m$-strongly connected subcomplex
on $D+m+1$ vertices from 
being counted more than once. 

Let $\eta\in \big(0,(1-\vep)/3\big)$. Denoting $a_m:= n^\taual/(m+2)!$
we use \eqref{e:decomp.Betti} and the above bounds to write 
\begin{align}
&\P\big( \beta_{m,n}\le (1-\vep) a_m \big) \ge \P \Big( S_{m+2}\le (1-\vep -2\eta)a_m,  \label{e:change.to.comb} \\
&\qquad \qquad\qquad \qquad \qquad \qquad  \sum_{j=m+3}^{D+m}\binom{j}{m+1}S_j \le \eta a_m, \, n^{m+1}R_{D+m+1}\le \eta a_m\Big). \notag 
\end{align}
Since the indicators in \eqref{e:def.Sj} and \eqref{e:def.Rj} are all increasing
functions of $\Gamma_\bp$, we can use the FKG inequality (see, e.g.,
\cite[Section 2.3]{meester:roy:1996}) to see that \eqref{e:change.to.comb} is bounded below by 
$$
 \P\big(S_{m+2} \le (1-\vep -2\eta)a_m \big)  \, \P  \Big( \sum_{j=m+3}^{D+m} \binom{j}{m+1}S_j \le \eta a_m \Big)\, \P (n^{m+1} R_{D+m+1}\le \eta a_m). 
$$

We claim that for every $j\ge m+3$, 
$$
\E[S_j] \le C_j n^\taual (n^{1-\psi_m})^{j-m-2}
$$
for a finite $C_j$ independent of $n$. 
Indeed, by \eqref{e:def.Sj}, 
$$ 
\E[S_j] =\sum_{K: |K|=j} \binom{n}{j} 
\P \Big( \text{$X(\sigma_j,
  \bp)$  contains a subcomplex isomorphic to $K$} \Big), 
$$
where the sum above is taken over all isomorphism classes of $m$-strongly connected complexes on $j$ vertices, with at least one trivial or non-trivial $m$-cycles. 
Here, and in similar situations in the sequel, we use a notation of
the type $\sigma_j$ to denote any fixed word  of length $j$ in $[n]$;
this is legitimate since the indicator functions in
\eqref{e:def.Sj} all have the same expectations. 
By
Lemma \ref{l:estimate.st.conn} and the fact that  the  number of terms
in the above sum   depends only on $j$, 
$$
\E[S_j] \le C_j n^j \prod_{i=q}^m p_i^{\binom{m+2}{i+1}} \Big(  \prod_{i=q}^m p_i^{\binom{m}{i}}\Big)^{j-m-2} =C_j n^\taual (n^{1-\psi_m})^{j-m-2}. 
$$
By Markov's inequality, 
\begin{align}
\begin{split}  \label{e:2nd.Markov}
\P  \Big( \sum_{j=m+3}^{D+m} \binom{j}{m+1}S_j \le \eta a_m \Big) &\ge 1-\eta^{-1} a_m^{-1} \sum_{j=m+3}^{D+m} C_j n^\taual (n^{1-\psi_m})^{j-m-2} \\ 
&\ge 1-C^* n^{1-\psi_m} \to 1, \ \ \text{ as } n\to\infty. 
\end{split}
\end{align}
Next, we claim  that 
\begin{equation}  \label{e:upper.bdd.ER}
\E[R_{D+m+1}] \le C^* n^{\tau_m-D(\psi_m-1)}. 
\end{equation}
To see this, note that by \eqref{e:def.Rj},
\begin{align}  \label{e:upper.bdd.E.R.D+m+1.1}
\E[R_{D+m+1}] =& \sideset{}{'}\sum_{K: |K|=D+m+1} \binom{n}{D+m+1}  \\
\notag &\hskip 0.8in \times \P \Big(  \text{$X(\sigma_{D+m+1},
  \bp)$  contains a subcomplex isomorphic to $K$} \Big), 
\end{align}
where the above sum is taken over all isomorphism classes of $m$-strongly connected complexes on $D+m+1$ vertices.
Since any simplicial complex  $K$ in the above sum  contains at least
$\binom{m+1}{i+1} + D\binom{m}{i}$ faces of dimension $i$ for each
$1\le i \le m$ (the first step of the argument in \cite[Lemma
8.1]{fowler:2019} is a detailed derivation of this fact),  
 we have for every $K$,
\begin{align}  \label{e:upper.bdd.E.R.D+m+1.2}
&\binom{n}{D+m+1} \P
\Big(  \text{$X(\sigma_{D+m+1},
  \bp)$ contains a subcomplex isomorphic to $K$} \Big)\\
\notag 
 & \hskip 2in \leq n^{D+m+1} \prod_{i=q}^m p_i^{\binom{m+1}{i+1} +
   D\binom{m}{i}} = n^{\tau_m -D(\psi_m-1)},  
\end{align}
and substituting  \eqref{e:upper.bdd.E.R.D+m+1.2}  into
\eqref{e:upper.bdd.E.R.D+m+1.1}  proves \eqref{e:upper.bdd.ER}. Now 
Markov's inequality and \eqref{e:upper.bdd.ER} show that
for large $n$, 
\begin{align}
\begin{split}  \label{e:3rd.Markov}
\P(n^{m+1}R_{D+m+1}\le \eta a_m) &\ge 1-n^{m+1} \E[R_{D+m+1}]\\
&\ge 1-C^* n^{m+1+\tau_m -D(\psi_m-1)} \to1, \ \ \ n\to\infty, 
\end{split}
\end{align}
where the last convergence is due to  the choice of $D$.  
The proof is completed  by combining  \eqref{e:2nd.Markov} and \eqref{e:3rd.Markov}. 
\end{proof}

Proposition \ref{p:lower.LD.leading.term2} below establishes  a lower
bound for the expression in the right hand side of 
\eqref{e:first.leading.term}.  Theorem \ref{t:lower.LD} will then
follow from   Propositions
\ref{p:lower.bdd.with.leading.term} and
\ref{p:lower.LD.leading.term2} by letting $\eta\downarrow0$.  We keep
the above notation $a_m=n^\taual/(m+2)!$.

For the proof of Proposition \ref{p:lower.LD.leading.term2}, we will exploit Theorem 2 in \cite{janson:warnke:2016} and Proposition \ref{p:small.q.face} of
the Appendix; for this purpose, we need to reformulate $S_{m+2}$ in the following way. For $\al \subset [n]$ with $|\al|=m+2$, denote by 
$$
\Big(\al_1^{(i)}, \al_2^{(i)},\dots,\al_{\binom{m+2}{i+1}}^{(i)}  \Big), \ \ i=q,\dots,m, 
$$
the collection of words of length $i+1$, that are contained in $\al$. Define 
\begin{equation}  \label{e:def.Qm2}
Q_{m+2}(\al) := \Big\{  \al_{p_2}^{(p_1)} : p_1=q,\dots,m, \, p_2=1,\dots,\binom{m+2}{p_1+1}\Big\}, 
\end{equation}
and $I_{\al}^{(m+2)} := \one \big\{ Q_{m+2}(\al)\subset \Gamma_\bp \big\}$; 
then, we have 
\begin{equation}  \label{e:reformulate.Sm2}
S_{m+2} = \sum_{\al\subset [n], \, |\al|=m+2} I_\al^{(m+2)}. 
\end{equation}

\begin{proposition} \label{p:lower.LD.leading.term2}  Assume that \eqref{e:pos.assump} holds. 
Let $\vep\in (0,1)$ and $\eta\in \big(0,(1-\vep)/3\big)$.  
\\
$(i)$ If $\taual\ge \tau_q$, then 
$$
\liminf_{n\to\infty} \frac{1}{n^{\tau_q}}\, \log \P\big( S_{m+2}  \le (1-\vep -2\eta) a_m \big) \ge -C_{m,q}  (\vep+3\eta)^2, 
$$
where $C_{m,q}$ is given in \eqref{e:def.C.epsilon}. \\
$(ii)$ If $\taual < \tau_q$, 
$$
\liminf_{n\to\infty} \frac{1}{n^\taual} \, \log \P \big(  S_{m+2}  \le (1-\vep -2\eta) a_m \big) \ge -\frac{5000}{(m+2)!}\, (\vep+3\eta)^2. 
$$
\end{proposition}
\begin{proof}
Set $X_m$ to be the right hand side in \eqref{e:reformulate.Sm2} and let $K$ be  either a trivial or non-trivial $m$-cycle on $m+2$ points. Clearly  $K$ contains
$\binom{m+2}{i+1}$ faces of dimension $i$ for each $1 \le i \le m$, 
and hence, 
\begin{equation}  \label{e:EXm}
\E[X_{m}] = \binom{n}{m+2}\prod_{i=q}^m p_i^{\binom{m+2}{i+1}} \sim \frac{n^\taual}{(m+2)!} = a_m,  \ \ \ n\to\infty. 
\end{equation}
Therefore, for large $n$ enough, 
$$
(1-\vep-3\eta)\E[X_{m}] \le (1-\vep-2\eta)a_m. 
$$
It thus suffices to lower bound the probability 
\begin{equation}  \label{e:adjusted.RHS}
\P\big( X_{m} \le (1-\vep-3\eta)\E[X_{m}] \big). 
\end{equation}
Set, for $\al, \beta \subset [n]$, $|\al|=|\beta|=m+2$, 
\begin{equation}  \label{e:equiv.rela2}
\al \sim \beta \ \text{ if and only if } \  Q_{m+2}(\al)\cap Q_{m+2}(\beta)\neq \emptyset,  \ \al\neq \beta.
\end{equation}

In order to apply Theorem 2 in \cite{janson:warnke:2016} and Proposition \ref{p:small.q.face} of
the Appendix, we have to deal with  the  following
quantities:
\begin{align}
\begin{split}  \label{e:4.quantities}
\mu(X_{m})&:=\E[X_{m}], \\
\Lambda(X_{m}) &:= \mu(X_m) + \sum_{\substack{\al, \beta \subset [n], |\al|=|\beta|=m+2, \\ \al\sim \beta}} \E[I_\al^{(m+2)}I_\beta^{(m+2)}],\\
\delta(X_m) &:= \Lambda(X_m)/\mu(X_m)-1, \\
\Pi(X_m) &:= \max_{\al \subset [n], \, |\al|=m+2} \E[I_\al^{(m+2)}] \
\text{(the biggest among identical values).}
\end{split}
\end{align}
The behaviour of $\mu(X_m)$ is given in \eqref{e:EXm}. Furthermore,  $\Pi(X_m)  =
n^{-\sum_{i=q}^m\binom{m+2}{i+1}\al_i}\to0$ as $n\to\infty$. 
The asymptotics of $\Lambda(X_m)$ and $\delta(X_m)$ are more involved and postponed to   Lemma \ref{l:Lambda.Delta} in the Appendix. 

Denote by $\mH_{q+1}$ the collection of words of length $q+1$ in
$[n]$. For $\al\in \mH_{q+1}$ let 
\begin{equation}  \label{e:def.J.q+1}
J_\al^{(q+1)} = \one \{ \al\subset \Gamma_\bp \}.
\end{equation}
We set 
\begin{equation}  \label{e:def.mU}
\U = \big[ \lfloor n/2\rfloor \big] = \{  1,2,\dots,\lfloor n/2\rfloor \},
\end{equation} 
and 
$$
H_{q+1}(\U) = \{ \al\in\mH_{q+1}: V(\al)\subset \U \}, 
$$
where $V(\al)$ denotes a vertex set of $\al$. Then 
\begin{equation}\label{e:def.Yq}
Y_q:= \sum_{\al\in \mH_{q+1}(\U)}J_\al^{(q+1)} 
\end{equation}
represents the $q$-face counts in $X(\U, \bp)$. Analogously to \eqref{e:4.quantities}, we also define 
\begin{align*}
\begin{split} 
\mu(Y_q)&:=\E[Y_q], \\
\Pi(Y_q) &:= \max_{\al\in \mH_{q+1}(\U)} \E[J_\al^{(q+1)}] \
\text{(again, the biggest among identical values).}
\end{split}
\end{align*}
Since \eqref{e:def.Yq} is the sum of i.i.d.~random variables,  one can
set $\Lambda(Y_q) \equiv \mu(Y_q)$ and $\delta(Y_q)\equiv 0$.  Clearly, 
\begin{equation}  \label{e:mu.Yq}
\mu(Y_q) = \binom{\lfloor n/2\rfloor}{q+1} p_q \sim \frac{n^{\tau_q}}{2^{q+1}(q+1)!}, \ \text{ and } \ \Pi(Y_q) = n^{-\al_q} \to 0. 
\end{equation}

\medskip

\noindent \textit{\underline{Case I}}: $\taual \ge \tau_q$. \\
Suppose first that   $\vep +3\eta\in [2^{-(m+5)}, 1]$. 
Then, 
\begin{align*}
\P\big( X_m \le (1-\vep-3\eta)\E[X_m] \big) &\ge \P(X_m=0) \ge \P \Big( \bigcap_{\al\in \mH_{q+1}} \{ J_\al^{(q+1)}=0 \} \Big). 
\end{align*} 
Since $\{ J_\al^{(q+1)} =0\}$ is a decreasing event for each $\al\in \mH_{q+1}$ in terms of the Bernoulli random variables defining $\Gamma_\bp$, the FKG inequality ensures that the last expression  is at least 
\begin{align*}
\prod_{\al\in \mH_{q+1}} \P(J_\al^{(q+1)}=0) &= (1-n^{-\al_q})^{\binom{n}{q+1}} \sim \exp \Big\{  -\frac{n^{\tau_q}}{(q+1)!} \Big\}  \ge \exp \Big\{  -\frac{2^{2(m+5)}(\vep+3\eta)^2 n^{\tau_q} }{(q+1)!} \Big\}. 
\end{align*}
It thus follows that
$$
\liminf_{n\to\infty} \frac{1}{n^{\tau_q}} \log \P \big( X_m \le (1-\vep -3\eta) \E[X_m] \big) \ge -\frac{2^{2(m+5)}(\vep+3\eta)^2 }{(q+1)!}. 
$$

Suppose next that $\vep+3\eta\in (0,2^{-(m+5)})$. By  Proposition \ref{p:small.q.face} with  $i=q$,
\begin{equation}  \label{e:need.prop.i=q}
\P \big( X_m \le (1-\vep -3\eta) \E[X_m] \big) \ge \frac{1}{6} \P\big( Y_q \le (1-2^{m+5}(\vep+3\eta)) \E[Y_q]\big). 
\end{equation}
Notice that $\Pi(Y_q)<1$ and for large $n$, by \eqref{e:mu.Yq}, 
$$
(\vep+3\eta)^2 \mu(Y_q) \ge \one \{ \vep + 3\eta < 1/50 \}\, \big(1+\delta(Y_q)\big)^{-1/2} = \one \{ \vep + 3\eta < 1/50 \}.
$$
Therefore, Theorem 2 in \cite{janson:warnke:2016}  is applicable, yielding that for large $n$,
\begin{align*}
&\frac{1}{6} \P\big( Y_q \le (1-2^{m+5}(\vep+3\eta)) \E[Y_q]\big)  \\
&\ge \frac{1}{6}  \exp \Big\{  -\frac{5000}{\big( 1-\Pi(Y_q) \big)^5}\, \big( 2^{m+5}(\vep+3\eta) \big)^2 \mu(Y_q) \big( 1+\delta(Y_q) \big) \Big\} \\
&=\exp  \Big\{ -\log 6 -\frac{5000 \cdot 2^{2(m+5)} (\vep+3\eta)^2}{(1-n^{-\al_q})^5} \, \mu(Y_q)\Big\}. 
\end{align*}
We conclude that 
$$
\liminf_{n\to\infty} \frac{1}{n^{\tau_q}} \log \P \big( X_m \le (1-\vep -3\eta) \E[X_m] \big) \ge -\frac{5^4 \cdot 2^{2(m+6)-q} (\vep+3\eta)^2}{(q+1)!}. 
$$
\medskip

\noindent \textit{\underline{Case II}}: $\taual<\tau_q$. \\
Unlike Case I, we  use \cite[Theorem 2]{janson:warnke:2016} only. By Lemma
\ref{l:Lambda.Delta} and  \eqref{e:EXm}, it holds  that $\Pi(X_m)<1$
and for large $n$, 
$$
(\vep+3\eta)^2 \mu(X_m) \ge \one \{ \vep + 3\eta < 1/50 \}\, \big(1+\delta(X_m)\big)^{-1/2}. 
$$
Hence, Theorem 2 in \cite{janson:warnke:2016} shows  that 
\begin{align*}
&\liminf_{n\to\infty} \frac{1}{n^{\taual}} \log \P \big( X_m \le (1-\vep -3\eta) \E[X_m] \big) \\
&\ge \liminf_{n\to\infty} \frac{1}{n^{\taual}}  \bigg\{ -\frac{5000}{\big( 1-\Pi(X_m) \big)^5}\, (\vep+3\eta)^2 \mu(X_m) \big( 1+\delta(X_m) \big) \bigg\} \\
&=-\frac{5000}{(m+2)!}\, (\vep+3\eta)^2. 
\end{align*}
\end{proof}

\subsection{Proof of Theorem \ref{t:SLLN}}

We use the representation in \eqref{e:repre.Betti}. We begin with a
proposition showing that $T_{m+2}$, the first term  in
\eqref{e:repre.Betti}, itself satisfies a SLLN with the same
scaling, and the same limit, as stated in Theorem \ref{t:SLLN}. After
that, Propositions \ref{p:SLLN.T.jr} and \ref{p:SLLN.residual} verify
that the rest of the terms in \eqref{e:repre.Betti} are all negligible
under the same scaling. 

\begin{proposition}  \label{p:SLLN.T.m+2} 
Under the assumption \eqref{e:pos.assump} we have as $n\to\infty$, 
$$
\frac{T_{m+2}}{n^\taual} \to \frac{1}{(m+2)!} \ \ \text{a.s.}
$$
\end{proposition}

The main difficulty here is that the power of $n$ in the scaling
constant of  $T_{m+2}$ above may not be large enough for a standard  use of
the Borel-Cantelli lemma.  We overcome this issue using a known
technique of concentrating first on 
subsequences increasing faster than linearly fast. In our context this
requires properly bounding $T_{m+2}$ 
from above and below first.  
This approach is a higher-dimensional analog of a technique in Chapter
3 of the monograph \cite{penrose:2003}, which  can  be  used, for
example, to show the SLLN for subgraph counts in random geometric
graphs. There are several related works \cite{owada:2022,
  owada:wei:2022} that used  this method to show SLLNs for topological
invariants of a random geometric complex. 

The proof of Proposition \ref{p:SLLN.T.m+2} requires some setup
first. Let  $v_N=N^\gamma$ for $\gamma =1,2,\dots$ such that  
\begin{equation}  \label{e:gamma.constraint}
\gamma(\taual)>1. 
\end{equation}
Then, for every $n\ge1$, there exists a unique $N=N(n)\in \bbn$ such
that $v_N\le n < v_{N+1}$. For each $i\ge1$, let 
$$  
p_i^\uparrow = v_N^{-\al_i}, \ \  \text{ and } \ \ p_i^\downarrow =v_{N+1}^{-\al_i}. 
$$

Now, we are ready to set up the desired  bounds for
$T_{m+2}$.  We define 
\begin{equation}  \label{e:def.S.m+2.uparrow}
S_{m+2}^\uparrow  =\sum_{\sigma\subset [v_{N+1}], \, |\sigma|=m+2} \xi_\sigma^{(m+2,m,\uparrow)}, 
\end{equation}
where $\xi_\sigma^{(m+2,m,\uparrow)}$ is the indicator of the event on
 which   the complex  $X\big(
\sigma, p_1^\uparrow, \dots, p_m^\uparrow, p_{m+1}^\uparrow \big)$  
forms a trivial or non-trivial $m$-cycle (since
\eqref{e:def.S.m+2.uparrow} counts the number of  $m$-cycles, there is
no need to specify connectivity probabilities of dimension greater
than $m+1$).  
Similarly, we define 
\begin{equation}  \label{e:def.T.m+2.downarrow}
T_{m+2}^\downarrow = \sum_{\sigma\subset [v_N], \, |\sigma|=m+2} \eta_\sigma^{(m+2,1,m,\downarrow)}. 
\end{equation}
Here, $\eta_\sigma^{(m+2,1,m,\downarrow)}$ is the indicator function of the event 
on which   the complex $X\big(
\sigma, p_1^\downarrow, \dots, p_m^\downarrow, p_{m+1}^\uparrow \big)$
forms  a non-trivial $m$-cycle 
and no vertex $s\in
[v_{N+1}]\setminus \sigma$ forms an $m$-face with an $(m-1)$-face in
$\sigma$, under the connectivity probabilities $(p_1^\uparrow, \dots,
p_m^\uparrow)$.  

Given random variables $X$ and $Y$, denote by $X\stackrel{st}{\le}Y$
the stochastic domination meaning that $\P(X\le x)\ge \P(Y\le x)$ for
all $x\in \R$. Recalling \eqref{e:def.T.jr} and
\eqref{e:def.Sj}, it   is clear that 
\begin{equation}  \label{e:stoha.domination1}
S_{m+2} \stackrel{st}{\le} S_{m+2}^\uparrow, \ \ \text{ and } \ \ T_{m+2}^\downarrow\stackrel{st}{\le} T_{m+2}. 
\end{equation}
 
\begin{proof}[Proof of Proposition \ref{p:SLLN.T.m+2}]
According to the stochastic domination in \eqref{e:stoha.domination1}, 
$$
\frac{T_{m+2}^\downarrow}{v_{N+1}^\taual} \stackrel{st}{\le} \frac{T_{m+2}}{n^\taual} \stackrel{st}{\le} \frac{S_{m+2}^\uparrow}{v_N^\taual}, 
$$
for every $n\ge1$. By the first Borel-Cantelli lemma, the claim of the
proposition will follow once  we show that for every $\vep>0$, 
\begin{align}
&\sum_{N=1}^\infty \P\Big( S_{m+2}^\uparrow \ge \frac{(1+\vep)^2}{(m+2)!}\, v_N^\taual \Big) <\infty,\label{e:BC1}\\
&\sum_{N=1}^\infty \P\Big( T_{m+2}^\downarrow \le  \frac{(1+\vep)^{-2}}{(m+2)!}\, v_{N+1}^\taual   \Big) <\infty. \label{e:BC2}
\end{align}
For the proof of \eqref{e:BC1},  we have 
$$
\E[S_{m+2}^\uparrow] = \binom{v_{N+1}}{m+2}\prod_{i=q}^m (p_i^\uparrow)^{\binom{m+2}{i+1}}  = \binom{v_{N+1}}{m+2}v_N^{-\sum_{i=q}^m \binom{m+2}{i+1}\al_i}. 
$$
As $v_{N+1}/v_N\to 1$ as $N\to\infty$, we have 
$$
\E[S_{m+2}^\uparrow] \sim \frac{v_N^\taual}{(m+2)!}, \ \ \ N\to\infty, 
$$
hence, there exists $N_0\in \bbn$ such that for all $N\ge N_0$, 
$$
\E[S_{m+2}^\uparrow] \le \frac{1+\vep}{(m+2)!}\, v_N^\taual. 
$$
By  Chebyshev's inequality and \eqref{e:Var.S.m+2.uparrow} in Lemma \ref{l:variance.asym2} of  the Appendix, 
\begin{align*}
\sum_{N=1}^\infty \P\Big( S_{m+2}^\uparrow \ge \frac{(1+\vep)^2}{(m+2)!}\, v_N^\taual \Big) &\le \sum_{N=1}^\infty \P\Big( S_{m+2}^\uparrow-\E[S_{m+2}^\uparrow] \ge \frac{\vep (1+\vep)}{(m+2)!}\, v_N^\taual \Big) \\
&\le C^* \sum_{N=1}^\infty \frac{\text{Var}(S_{m+2}^\uparrow)}{v_N^{2(\taual)}}  \\
&\le C^* \sum_{N=1}^\infty  \frac{v_{N+1}^{2(\taual)-\tau_q}\vee v_{N+1}^\taual}{v_N^{2(\taual)}} \\
&\le C^*  \sum_{N=1}^\infty \big( N^{-\gamma \tau_q} \vee N^{-\gamma (\taual)} \big). 
\end{align*}
By \eqref{e:gamma.constraint} and  the fact that $\tau_q>q\ge 1$, the last sum is  finite. 

We now prove \eqref{e:BC2}. We claim that one can write 
\begin{equation}  \label{e:E.T.m+2.downarrow}
\E[T_{m+2}^\downarrow] = \binom{v_N}{m+2}\prod_{i=q}^m (p_i^\downarrow)^{\binom{m+2}{i+1}} (1-p_{m+1}^\uparrow) \Big( 1-\binom{m+2}{m}v_{N}^{-\psi_m} + u_N \Big)^{v_{N+1}-(m+2)}, 
\end{equation}
where $u_N=o(v_{N}^{-\psi_m})$ 
as $N\to\infty$.  Indeed, for every $\sigma \subset [v_N]$ with
$|\sigma|=m+2$, the probability that   $X\big( \sigma,
p_1^\downarrow,\dots,p_m^\downarrow,p_{m+1}^\uparrow \big)$ forms a
non-trivial $m$-cycle is given by $\prod_{i=q}^m (p_i^\downarrow)^{\binom{m+2}{i+1}} (1-p_{m+1}^\uparrow)$. Further, $\sigma$ contains $\binom{m+2}{m}$ faces of dimension $m-1$. 
Consider a vertex $w\in [v_{N+1}]\setminus \sigma$ and an $(m-1)$-face $f$ in $\sigma$. Then, $w$ and $f$ form an $m$-face with probability  
$\prod_{i=q}^m (p_i^\uparrow)^{\binom{m}{i}} = v_{N}^{-\psi_m}$. 
By the inclusion-exclusion formula, $w$ forms an $m$-face with at
least one of the $(m-1)$-faces in $\sigma$, with probability
$\binom{m+2}{m}v_{N}^{-\psi_m}- u_N$, with $u_N$ accounting for the
probability that multiple $m$-faces are formed between $w$ and the
$(m-1)$-faces in $\sigma$. Since there are at least $v_{N+1}-(m+2)$ vertices
in $[v_{N+1}]\setminus \sigma$,  the probability that no vertex $s\in
[v_{N+1}]\setminus \sigma$ forms an $m$-face with an $(m-1)$-face in
$\sigma$ can be written as 
\begin{equation}  \label{e:maximality.sigma}
\Big( 1-\binom{m+2}{m}v_{N}^{-\psi_m} + u_N \Big)^{v_{N+1}-(m+2)}, 
\end{equation}
and \eqref{e:E.T.m+2.downarrow} follows. 

As $\psi_m>1$ and $\al_{m+1}>0$, we see that \eqref{e:maximality.sigma} tends to $1$ as $N\to\infty$, and $1-p_{m+1}^\uparrow\to1$, $N\to\infty$. Therefore, 
$$
\E[T_{m+2}^\downarrow] \sim\frac{v_N^{m+2}}{(m+2)!}\, v_{N+1}^{-\sum_{i=q}^m \binom{m+2}{i+1}\al_i}\sim \frac{v_{N+1}^\taual}{(m+2)!}, \ \ \ N\to\infty, 
$$
and so, there exists $N_1\in \bbn$ such that for all $N\ge N_1$, 
$$
\E[T_{m+2}^\downarrow] \ge \frac{(1+\vep)^{-1}}{(m+2)!}\, v_{N+1}^\taual.
$$
Once again, by Chebyshev's inequality and \eqref{e:Var.T.m+2.downarrow} of Lemma \ref{l:variance.asym2} conclude that 
\begin{align*}
\sum_{N=1}^\infty \P\Big( T_{m+2}^\downarrow \le  \frac{(1+\vep)^{-2}}{(m+2)!}\, v_{N+1}^\taual   \Big) &\le \sum_{N=1}^\infty \P\Big( T_{m+2}^\downarrow  - \E[T_{m+2}^\downarrow ]\le  -\frac{\vep(1+\vep)^{-2}}{(m+2)!}\, v_{N+1}^\taual   \Big)\\
&\le C^*\sum_{N=1}^\infty \big(N^{-\gamma \tau_q}\vee N^{-\gamma (\taual)}  \big) <\infty, 
\end{align*}
as required. 
\end{proof}

In the next proposition we start addressing the rest of the terms in the representation
\eqref{e:repre.Betti} of the Betti numbers. 
 
\begin{proposition}   \label{p:SLLN.T.jr}
Assume \eqref{e:pos.assump}. Then for every $j\ge m+3$ and $r\ge1$, 
$$
\frac{T_{j,r}}{n^\taual} \to 0 \ \ \text{a.s. as } \ n\to\infty.
$$
\end{proposition}

\begin{proof}
  
Analogously to \eqref{e:def.S.m+2.uparrow} we define 
\begin{equation}  \label{e:def.Sj2.uparrow}
S_j^{\uparrow} := \sum_{\sigma\subset [v_{N+1}], \, |\sigma|=j}
\xi_\sigma^{(j,m, \uparrow)}, \ \ \ j \ge m+3, 
\end{equation}
where, $\xi_\sigma^{(j,m, \uparrow)}$ is the indicator of
the event on which  $X\big( \sigma,
p_1^\uparrow,\dots,p_m^\uparrow, p_{m+1}^\uparrow \big)$ contains an
$m$-strongly connected subcomplex  with at least 
one trivial or non-trivial $m$-cycle.  Clearly,  
$S_j\stackrel{st}{\le} S_j^{\uparrow}$ for all $j\ge m+3$, so that 
$$
\frac{T_{j,r}}{n^\taual} \le \frac{S_j}{n^\taual}  \stackrel{st}{\le} \frac{S_j^\uparrow}{v_N^\taual}
$$
  for every  $n\ge1$ and $r\ge1$, and 
by the first Borel-Cantelli lemma, it is suffices to show that for every $\vep>0$, 
\begin{equation}  \label{e:BC3}
\sum_{N=1}^\infty \P(S_j^{\uparrow} \ge \vep v_N^{\taual})<\infty. 
\end{equation}
Note that 
$$ 
\E[S_j^{\uparrow}] = \sum_{K:|K|=j}\binom{v_{N+1}}{j} \P \big( X\big(
\sigma, p_1^\uparrow, \dots, p_m^\uparrow,p_{m+1}^\uparrow \big)
 \  \text{contains a subcomplex isomorphic to} \ K \big), 
$$
where the sum is taken over all isomorphism classes of $m$-strongly
connected complexes on $j$ vertices, with at least one trivial or
non-trivial $m$-cycle.  It follows from Lemma \ref{l:estimate.st.conn}
and  $v_{N+1}/v_N\to1$ as $N\to\infty$, that for a
$j$-dependent constant $C_j$, 
\begin{align*}
\E[S_j^{\uparrow}] &\le C_j v_{N+1}^j\prod_{i=q}^m (p_i^\uparrow)^{\binom{m+2}{i+1}} \Big\{  \prod_{i=q}^m (p_i^\uparrow)^{\binom{m}{i}}\Big\}^{j-m-2} \\
&\le C_j v_N^\taual \big(v_N^{1-\psi_m}\big)^{j-m-2} = o \big(v_N^\taual  \big),  \ \ \ N\to\infty, 
\end{align*}
where the last equality is due to $\psi_m>1$ (and $j\ge m+3$). Hence, for sufficiently large $N$, 
\begin{align}
\begin{split}  \label{e:Sj2.uparrow.summable}
\P(S_j^{\uparrow} \ge \vep v_N^{\taual}) &\le \P \Big( S_j^{\uparrow} - E[S_j^{\uparrow}] \ge \frac{\vep}{2} v_N^\taual \Big)   \\
&\le \frac{4}{\vep^2}\cdot \frac{\text{Var}(S_j^{\uparrow})}{v_N^{2(\taual)}} \le C_j \big( N^{-\gamma \tau_q}\vee N^{-\gamma (\taual)} \big), 
\end{split}
\end{align}
where the last inequality follows  from \eqref{e:Var.S.m+2.uparrow} in
Lemma \ref{l:variance.asym2} of the Appendix. The last term in
\eqref{e:Sj2.uparrow.summable} is summable, so \eqref{e:BC3}
follows.
\end{proof}

It follows from Propositions \ref{p:SLLN.T.m+2} and \ref{p:SLLN.T.jr} 
that 
$$
\frac{1}{n^\taual} \Big( T_{m+2} + \sum_{j=m+3}^{D+m}
\sum_{r\ge1}rT_{j,r} \Big) \to \frac{1}{(m+2)!} \ \ \text{a.s. as } \
n\to\infty, 
$$
where $D$ is a positive  integer satisfying \eqref{e:def.D}.  The
proof of Theorem \ref{t:SLLN}  is now  completed by  Proposition
\ref{p:SLLN.residual} below and   \eqref{e:repre.Betti}. 

\begin{proposition}  \label{p:SLLN.residual}
Assume \eqref{e:pos.assump}. 
We have, as $n\to\infty$, 
$$
\frac{1}{n^\taual} \sum_{j=D+m+1}^n \sum_{r\ge1}rT_{j,r} \to 0, \ \ \ \text{a.s.}
$$
\end{proposition}
\begin{proof}
Repeating the same argument as in \eqref{e:upper.bdd.E.R.D+m+1.1} and
\eqref{e:upper.bdd.E.R.D+m+1.2}, along with Markov's inequality, we
have, 
\begin{align*}
&\sum_{n=1}^\infty \P\Big( \sum_{j=D+m+1}^n \sum_{r\ge1}rT_{j,r} \ge
  1 \Big) \le   
\sum_{n=1}^\infty \sum_{j=D+m+1}^n \sum_{r\ge1} r\E[T_{j,r}]  \\
&\le C^* \sum_{n=1}^\infty n^{m+1+\tau_m-D(\psi_m-1)} < \infty
\end{align*}
by the choice of $D$. 
The first Borel-Cantelli lemma completes the proof of Proposition \ref{p:SLLN.residual}. 
\end{proof}
\medskip

\subsection{Proof of Theorem \ref{t:CLT}}

Theorem \ref{t:CLT} will be established by combining a series of
propositions. We start with  a  CLT for $S_{m+2}$, a special case of \eqref{e:def.Sj}. 

\begin{proposition}  \label{p:CLT.S.m+2}
Assume conditions \eqref{e:pos.assump} and \eqref{e:CLT.cond}.  As $n\to\infty$, 
$$
\frac{S_{m+2}-\E[S_{m+2}]}{\sqrt{\text{Var}(S_{m+2})}} \Rightarrow\mathcal N(0,1). 
$$
\end{proposition}
\begin{proof}
Our proof takes an approach similar  to that of Proposition 3.6 in
\cite{owada:samorodnitsky:thoppe:2021}.  
Using \eqref{e:reformulate.Sm2} we write 
$$
W=: \frac{S_{m+2}-\E[S_{m+2}]}{\sqrt{\text{Var}(S_{m+2})}}
 =\sum_{\sigma \subset [n], \,
  |\sigma|=m+2} Z_\sigma,
$$
where 
$$
Z_\sigma=\frac{I_\sigma^{(m+2)} - \E[I_\sigma^{(m+2)}] }{\sqrt{\text{Var}(S_{m+2})}}. 
$$
Evidently we have $\E[W]=0$, $\text{Var}(W)=1$, and further, $(Z_\sigma)$ constitute a
\emph{dissociated set} of random variables in the sense of Barbour et
al.~\cite{barbour:karonski:rucinski:1989}. Namely, as in \eqref{e:def.Qm2}, we identify each $\sigma \subset [n], \,
  |\sigma|=m+2$ with the collection $Q_q(\sigma)$ of its subsets of
  size $q+1$. Then for every two collections 
  $K, L$ of subsets of $[n]$ of length $m+2$, if
$$
\bigcup_{\sigma\in K}  Q_q(\sigma) \cap \bigcup_{ \sigma \in L}  Q_q(\sigma) =\emptyset, 
$$
then $(Z_\sigma, \, \sigma\in K)$ and $(Z_\sigma, \,  \sigma\in L)$
are independent.
The weak convergence claimed in the proposition 
will follow once we show that  
  the $L_1$-Wasserstein distance between the laws of $W$ and the
  standard normal random variable $Y$,   
\begin{equation}  \label{e:Wasserstein.metric}
d_1(W,Y):=\sup_{\phi:\R\to\R} \Big| \,  \E\big[ \phi(W) \big] - \E\big[ \phi(Y) \big]  \,\Big| \to 0, \ \ \ n\to\infty, 
\end{equation}
 where  the supremum is taken over all $1$-Lipschitz functions
 $\phi:\R\to\R$. 

For $\al \subset [n], \,
  |\al|=m+2$, let $\mL_\al :=\{\beta \subset [n], \,
  |\beta|=m+2: \, |\beta \cap \al|\ge q+1  \}$ be the dependency
  neighborhood of $\al$. By Theorem 1 in  Barbour et
  al.~\cite{barbour:karonski:rucinski:1989} (see also  the bound (2.7) therein), the distance \eqref{e:Wasserstein.metric} is bounded by a constant multiple of 
\begin{align}
\begin{split}  \label{e:Wasserstein.bound}
&\sum_{\al \subset [n], \,
  |\al|=m+2} \sum_{\beta\in \mL_\al} \sum_{\gamma\in \mL_\al} \Big\{ \E\big[ |Z_\al Z_\beta Z_\gamma| \big]  + \E\big[ |Z_\al Z_\beta| \big] \, \E\big[|Z_\gamma|  \big]\Big\}  \\
&\le \frac{C^*}{\big(\text{Var}(S_{m+2})\big)^{3/2}}\sum_{\al \subset [n], \,
  |\al|=m+2} \sum_{\beta\in \mL_\al}   \sum_{\gamma\in \mL_\al} \Big\{  \E \big[ (I_\al^{(m+2)} + \E[I_\al^{(m+2)}])  (I_\beta^{(m+2)} + \E[I_\beta^{(m+2)}]) \\
&\qquad \qquad   \times (I_\gamma^{(m+2)} + \E[I_\gamma^{(m+2)}]) \big] +2 \E \big[ (I_\al^{(m+2)} + \E[I_\al^{(m+2)}]) (I_\beta^{(m+2)} + \E[I_\beta^{(m+2)}]) \big]\, \E[I_\gamma^{(m+2)}]  \Big\}.   
\end{split}
\end{align}
Given  $ \al \subset [n], \, |\al|=m+2$, $\beta\in \mL_\al$ and $\gamma\in \mL_\al$, denote 
$$
\ell_{12} =|\al\cap \beta|, \ \ \ell_{13} =|\al\cap\gamma|, \ \ \ell_{23}=|\beta \cap \gamma|, \ \ \ell_{123}=|\al\cap \beta \cap \gamma|. 
$$
Since $\beta\in \mL_\al$ and $\gamma\in \mL_\al$, it must be that $\ell_{12}\ge q+1$, $\ell_{13}\ge q+1$, whereas $\ell_{23}$, $\ell_{123}$ can be less than $q+1$. 
\medskip

\noindent \textit{\underline{Part I}}: We first deal with the case
when $(\al,\beta,\gamma) $  satisfies $\ell_{12} \le m+1$, $\ell_{13}
\le m+1$,  $\ell_{23} \le m+1$.   Then 
\begin{equation}  \label{e:3rd.moment.I}
\E[I_\al^{(m+2)}I_\beta^{(m+2)}I_\gamma^{(m+2)}] =  \prod_{i=q}^m
p_i^{3\binom{m+2}{i+1} -
  \binom{\ell_{12}}{i+1}-\binom{\ell_{13}}{i+1}-\binom{\ell_{23}}{i+1}+\binom{\ell_{123}}{i+1}},  
\end{equation}
and all other expectation terms in \eqref{e:Wasserstein.bound} are
bounded by the right hand side of  \eqref{e:3rd.moment.I}.  
Appealing to  these bounds and  using the fact that the number of ways
to choose $(\al,\beta,\gamma) $ corresponding to given $\ell_{12},
\ell_{13}, \ell_{23}, \ell_{123}$ is bounded by $
n^{3(m+2)-\ell_{12}-\ell_{13}-\ell_{23}+\ell_{123}}$, 
one can  bound the corresponding terms in \eqref{e:Wasserstein.bound}, up to a constant factor, by 
$$
\frac{1}{\big(\text{Var}(S_{m+2})\big)^{3/2}}\sum_{\ell_{12}=q+1}^{m+1} \sum_{\ell_{13}=q+1}^{m+1} \sum_{\ell_{23}=0}^{m+1} \sum_{\ell_{123}=0}^{\ell_{12}\wedge \ell_{13}\wedge \ell_{23}} n^{3(\taual) - \tau_{\ell_{12}-1}- \tau_{\ell_{13}-1}- \tau_{\ell_{23}-1} +  \tau_{\ell_{123}-1}}. 
$$
If, say, $\ell_{23}\le k+1$; then,
$\tau_{\ell_{23}-1}-\tau_{\ell_{123}-1}\ge 0$ (see
\eqref{e:tau.maximized.at.crit}). Therefore, 
\begin{align*}
&\sum_{\ell_{12}=q+1}^{m+1} \sum_{\ell_{13}=q+1}^{m+1} \sum_{\ell_{23}=0}^{k+1} \sum_{\ell_{123}=0}^{\ell_{12}\wedge \ell_{13}\wedge \ell_{23}} n^{3(\taual) - \tau_{\ell_{12}-1}- \tau_{\ell_{13}-1}- \tau_{\ell_{23}-1} +  \tau_{\ell_{123}-1}} \\
&\le C^*\sum_{\ell_{12}=q+1}^{m+1} \sum_{\ell_{13}=q+1}^{m+1} n^{3(\taual)-\tau_{\ell_{12}-1}-\tau_{\ell_{13}-1}} \\
&\le C^* \big\{ n^{3(\taual)-2\tau_q}\vee n^{3(\taual)-2\tau_m} \big\},
\end{align*}
and similar bounds apply when $\ell_{12}\le k+1$ or $\ell_{13}\le
k+1$. Therefore,  it is  sufficient to bound 
\begin{align}
\begin{split}  \label{e:part1.bound}
&\frac{1}{\big(\text{Var}(S_{m+2})\big)^{3/2}} \Big\{ n^{3(\taual)-2\tau_q}\vee n^{3(\taual)-2\tau_m} \\
&\qquad + \sum_{\ell_{12}=k+2}^{m+1} \sum_{\ell_{13}=k+2}^{m+1} \sum_{\ell_{23}=k+2}^{m+1} \sum_{\ell_{123}=0}^{\ell_{12}\wedge \ell_{13}\wedge \ell_{23}} n^{3(\taual) - \tau_{\ell_{12}-1}- \tau_{\ell_{13}-1}- \tau_{\ell_{23}-1} +  \tau_{\ell_{123}-1}}  \Big\}. 
\end{split}
\end{align} 
By \eqref{e:taum.and.taum+1} and Lemma \ref{l:variance.asym1} in the
Appendix, the contribution of the first term in \eqref{e:part1.bound}  
is bounded by $C^* \big(n^{-\tau_q/2} \vee n^{-\frac{1}{2}(\taual)}\big)$, which clearly converges to $0$ as $n\to\infty$. Moreover, the second line of  \eqref{e:part1.bound} is bounded by 
$$
C^*\sum_{\ell_{12}=k+2}^{m+1} \sum_{\ell_{13}=k+2}^{m+1} \sum_{\ell_{23}=k+2}^{m+1} \sum_{\ell_{123}=0}^{\ell_{12}\wedge \ell_{13}\wedge \ell_{23}} n^{\frac{3}{2}(\tau_q\wedge (\taual)) - \tau_{\ell_{12}-1}- \tau_{\ell_{13}-1}- \tau_{\ell_{23}-1} +  \tau_{\ell_{123}-1}}, 
$$
which also tends to $0$ as $n\to\infty$ by the assumption \eqref{e:CLT.cond}. 
\medskip

\noindent \textit{\underline{Part II}}:  Next we deal with the case
when $(\al,\beta,\gamma) $  satisfies 
 one of the following conditions. 
\vspace{5pt}

\begin{itemize}
\item $\ell_{12}=m+2$, $\ell_{13}\vee \ell_{23}\le m+1$, 
\item $\ell_{13}=m+2$, $\ell_{12}\vee \ell_{23}\le m+1$, 
\item $\ell_{23}=m+2$, $\ell_{12}\vee \ell_{13}\le m+1$. 
\end{itemize}
\vspace{5pt}

\noindent  We treat in detail only the first case, as the other two
cases are similar.  The conditions imply that $\al=\beta$, and  it
follows that  $\ell_{13}=\ell_{23}=\ell_{123}$. Then the contribution
of the corresponding terms in \eqref{e:Wasserstein.bound} is, up to
a constant multiplier, bounded by  
\begin{align*}
&\frac{1}{\big( \text{Var}(S_{m+2})
                 \big)^{3/2}}\sum_{\ell_{13}=q+1}^{m+1} \sum_{ \al \subset [n], \,
  |\al|=m+2} 
\sum_{\gamma\in \mL_\al:\, |\al\cap\gamma|=\ell_{13}} \E [I_\al^{(m+2)}I_\gamma^{(m+2)}] \\
&=\frac{1}{\big( \text{Var}(S_{m+2})     \big)^{3/2}}
 \sum_{\ell_{13}=q+1}^{m+1} \sum_{ \al \subset [n], \,
  |\al|=m+2} 
\sum_{\gamma\in \mL_\al:\, |\al\cap\gamma|=\ell_{13}} 
\prod_{i=q}^m p_i^{2\binom{m+2}{i+1}-\binom{\ell_{13}}{i+1}}   \\
&\le \frac{1}{\big( \text{Var}(S_{m+2}) \big)^{3/2}}\sum_{\ell_{13}=q+1}^{m+1} n^{2(\taual)-\tau_{\ell_{13}-1}}  \\
&\le \frac{C^*}{\big( \text{Var}(S_{m+2}) \big)^{3/2}} \big\{ n^{2(\taual)-\tau_q} \vee n^{2(\taual)-\tau_m} \big\}  \\
&\le C^* \big\{ n^{-\tau_q/2} \vee n^{-\frac{1}{2}(\taual)} \big\} \to 0, \ \ \text{as } n\to\infty. 
\end{align*}

\noindent \textit{\underline{Part III}}:  
It remains to consider the case when $(\al,\beta,\gamma)$ satisfies $\ell_{12}=\ell_{13}=\ell_{23}=m+2$. 
Then $\alpha=\beta=\gamma$ and $\ell_{123}=m+2$, and the contribution
of the corresponding terms in \eqref{e:Wasserstein.bound} is bounded,
up to a constant multiplier, by 
\begin{align*}
\frac{1}{\big(\text{Var}(S_{m+2})\big)^{3/2}} \sum_{\al \subset [n], \,
  |\al|=m+2}\E[I_\al^{(m+2)}]  &\le \frac{C^* n^\taual}{\big\{ n^{2(\taual)-\tau_q}\vee n^\taual \big\}^{3/2}} \\
&\le C^*n^{-\frac{1}{2}(\taual)} \to 0, \ \ \text{as } n\to\infty. 
\end{align*}
\end{proof}

In the next proposition we  show that the   CLT for $S_{m+2}$  can be
translated into a CLT for $T_{m+2}$, a special case of \eqref{e:def.T.jr}. 

\begin{proposition}  \label{p:CLT.T.m+2}
Assume conditions \eqref{e:pos.assump} and \eqref{e:CLT.cond}.  As $n\to\infty$, 
$$
\frac{T_{m+2}-\E[T_{m+2}]}{\sqrt{\text{Var}(S_{m+2})}} \Rightarrow\mathcal N(0,1). 
$$
\end{proposition}
\begin{proof}
Define 
\begin{equation}  \label{e:def.V.m+2}
V_{m+2} :=\sum_{\sigma \subset [n], \, |\sigma|=m+2} \lambda_\sigma^{(m+2,m)}, 
\end{equation}
where $\lambda_\sigma^{(m+2,m)}$ is the indicator function of the
event that $X(\sigma,\bp)$ forms a non-trivial $m$-cycle. 
By virtue of the last proposition, it is enough to verify that, as $n\to\infty$, 
$$
\frac{S_{m+2}-T_{m+2} - \E[S_{m+2}-T_{m+2}]}{\sqrt{\text{Var}(S_{m+2})}} \stackrel{p}{\to}0. 
$$
For every $\vep>0$, Chebyshev's inequality gives that 
\begin{align}
\begin{split}  \label{e:after.Chebyshev}
&\P \big( \, |S_{m+2}-T_{m+2} - \E[S_{m+2}-T_{m+2}]| \ge \vep \sqrt{\text{Var}(S_{m+2})}\big)   \\
&\le \P \Big( \, |S_{m+2}-V_{m+2} - \E[S_{m+2}-V_{m+2}]| \ge \frac{\vep}{2} \sqrt{\text{Var}(S_{m+2})}\Big) \\
&\qquad + \P \Big( \, |V_{m+2}-T_{m+2} - \E[V_{m+2}-T_{m+2}]| \ge \frac{\vep}{2} \sqrt{\text{Var}(S_{m+2})}\Big)\\
&\le \frac{4}{\vep^2}\, \Big\{ \frac{\text{Var}(S_{m+2}-V_{m+2})}{\text{Var}(S_{m+2})}+\frac{\text{Var}(V_{m+2}-T_{m+2})}{\text{Var}(S_{m+2})} \Big\}. 
\end{split}
\end{align}
By definition, $S_{m+2}-V_{m+2}$ equals the number of  $(m+1)$-faces  in $X([n],\bp)$, so by Proposition 3.1 in \cite{owada:samorodnitsky:thoppe:2021}, 
\begin{equation}  \label{e:Var.S.m+2-V.m+2}
\text{Var}(S_{m+2}-V_{m+2})\le C^* (n^{2\tau_{m+1}-\tau_q}\vee n^{\tau_{m+1}}). 
\end{equation}
It now follows from \eqref{e:Var.S.m+2} and \eqref{e:Var.V.m+2-T.m+2}
of  Lemma \ref{l:variance.asym1}, as well as
\eqref{e:Var.S.m+2-V.m+2}, that the final expression in
\eqref{e:after.Chebyshev} is at most  
\begin{align*}
&\frac{C^*}{n^{2(\taual)-\tau_q}\vee n^{\taual}}\, \Big\{  n^{2\tau_{m+1}-\tau_q}\vee n^{\tau_{m+1}} \\
&\qquad \qquad\qquad \qquad\qquad\qquad + \big( n^{2(\taual)-\tau_q} \vee n^\taual \big) n^{(1\vee \psi_{m-1})-\psi_m}\Big\}  \\
&\le C^* (n^{-\al_{m+1}} + n^{(1\vee \psi_{m-1})-\psi_m}) \to 0, \ \ \text{as } n\to\infty. 
\end{align*}
\end{proof}

Finally, we show that the remaining terms in the expression 
\eqref{e:repre.Betti} of the Betti numbers are negligible in the
context of the required CLT. 

\begin{proposition}  \label{p:CLT.residual}
Assume \eqref{e:pos.assump} and \eqref{e:CLT.cond}. 
As $n\to\infty$, 
$$
\frac{\sum_{j=m+3}^n \sum_{r\ge1} rT_{j,r} - \E \big[ \sum_{j=m+3}^n \sum_{r\ge1} rT_{j,r} \big]}{\sqrt{\text{Var}(S_{m+2})}} \stackrel{p}{\to}0. 
$$
\end{proposition}
\begin{proof}
We first check that for every $j\ge m+3$ and $r\ge1$, 
$$
\frac{T_{j,r}-\E[T_{j,r}]}{\sqrt{\text{Var}(S_{m+2})}} \stackrel{p}{\to}0, \ \ \ n\to\infty. 
$$
Indeed, for  every $\vep>0$, Chebyshev's inequality, along with \eqref{e:Var.S.m+2} and \eqref{e:Var.T.jr} of Lemma \ref{l:variance.asym1},  demonstrate  that as $n\to\infty$,
\begin{align*}
\P \big(\, |T_{j,r}-\E[T_{j,r}]| \ge \vep\sqrt{\text{Var}(S_{m+2})}   \big)\le \frac{1}{\vep^2}\cdot \frac{\text{Var}(T_{j,r})}{\text{Var}(S_{m+2})} \le C_j n^{1-\psi_m} \to 0. 
\end{align*}

It remains to show that 
$$
\frac{\sum_{j=D+m+1}^n \sum_{r\ge1} rT_{j,r} - \E \big[
  \sum_{j=D+m+1}^n \sum_{r\ge1} rT_{j,r}
  \big]}{\sqrt{\text{Var}(S_{m+2})}} \stackrel{p}{\to}0, 
$$
where $D$ is a  positive constant satisfying
\eqref{e:def.D}. However, the denominator diverges to infinity, while
proceeding the  same way as in
\eqref{e:upper.bdd.E.R.D+m+1.1} and \eqref{e:upper.bdd.E.R.D+m+1.2}, 
one sees that the mean of  the absolute value of the numerator is
bounded  
by a constant multiple of $n^{m+1+\tau_m-D(\psi_m-1)}$, which goes to
$0$ as $n\to\infty$. This proves the claim. 
\end{proof}

\begin{proof}[Proof of Theorem \ref{t:CLT}]
By Propositions \ref{p:CLT.T.m+2}  and
\ref{p:CLT.residual}, 
along with the representation \eqref{e:repre.Betti}, 
$$
\frac{\beta_{m,n}-\E[\beta_{m,n}]}{\sqrt{\text{Var}(S_{m+2})}} \Rightarrow \mathcal N(0,1). 
$$
Now, \eqref{e:Var.S.m+2} in Lemma \ref{l:variance.asym1} completes the proof. 
\end{proof}
\medskip

\section{Appendix} 

We start with the asymptotic behaviour of $\Lambda(X_m)$ and $\delta(X_m)$
in \eqref{e:4.quantities}.  
 
\begin{lemma}\label{l:Lambda.Delta}
$(i)$ As $n\to\infty$, 
\begin{equation} \label{e:Lambda.Xm}
\Lambda(X_m) \sim \begin{cases}
\frac{n^{2(\taual)-\tau_q}}{(q+1)! ( (m-q+1)! )^2} & \text{ if } \taual>\tau_q, \\[7pt]
\frac{n^\taual}{(m+2)!} & \text{ if } \taual<\tau_q, \\[7pt]
\Big\{ \frac{1}{(q+1)! ( (m-q+1)! )^2}  + \frac{1}{(m+2)!} \Big\} n^\taual & \text{ if } \taual=\tau_q.
\end{cases}
\end{equation}
$(ii)$ As $n\to\infty$, 
\begin{equation}  \label{e:delta.Xm}
\delta(X_m) \sim \begin{cases}
\frac{(m+2)!n^{\taual-\tau_q}}{(q+1)! ( (m-q+1)! )^2} & \text{ if } \taual\ge \tau_q, \\[3pt]
0& \text{ if } \taual<\tau_q. 
\end{cases}
\end{equation}
\end{lemma}
\begin{proof}
For $\sigma \subset [n], |\sigma|=m+2$,  we use the notation
$\xi_\sigma:=\xi_\sigma^{(m+2,m)}$  in \eqref{e:def.Sj}. 
By  the definition of the relation  \eqref{e:equiv.rela2}  we have 
\begin{align*}
&\sum_{\substack{\al, \beta \subset [n], |\al|=|\beta|=m+2, \\ \al\sim \beta}} \E\big[ I_\al^{(m+2)}I_\beta^{(m+2)} \big]  \\
&=\sum_{\ell=q+1}^{m+1} \sum_{\substack{\sigma \subset [n], \\
  |\sigma|=m+2}}\sum_{\substack{\tau \subset [n], \\ |\tau|=m+2}}
  \E\big[ \xi_\sigma  \xi_\tau \big]\, \one \big\{ |\sigma \cap \tau|=\ell \big\}  \\
&=\sum_{\ell=q+1}^{m+1}    \binom{n}{m+2}\binom{m+2}{\ell}\binom{n-m-2}{m+2-\ell}
      \E\big[  \xi_{\sigma_{m+2} }  \xi_{\tau_{m+2}}   \big]\,
                             \one \big\{ |\sigma_{m+2} \cap \tau_{m+2}|=\ell \big\} \\
&\sim \sum_{\ell=q+1}^{m+1} \frac{n^{2(m+2)-\ell}}{\ell! \big(
                                                                                         (m+2-\ell)! \big)^2}\,  \E\big[  \xi_{\sigma_{m+2} }  \xi_{\tau_{m+2}}   \big]\,
                             \one \big\{ |\sigma_{m+2} \cap    \tau_{m+2}|=\ell \big\}. 
\end{align*}
For $\sigma,\tau \subset [n], \, |\sigma|=|\tau|=m+2, \, \ell=|\sigma \cap
\tau|\in \{q+1,\dots,m+1  \}$,  we have 
$$
\E\big[ \xi_\sigma \xi_\tau\big] =\prod_{i=q}^m p_i^{2\binom{m+2}{i+1}-\binom{\ell}{i+1}} = n^{-2\sum_{i=q}^m \binom{m+2}{i+1}\al_i +\sum_{i=q}^{\ell-1}\binom{\ell}{i+1}\al_i}, 
$$
where $\prod_{i=q}^m p_i^{-\binom{\ell}{i+1}}$ accounts for the double count of the $(\ell-1)$-face, which is common to $\sigma$ and $\tau$. Therefore, as $n\to\infty$, 
\begin{align}
\begin{split}  \label{e:cross.moment}
\sum_{\substack{\al, \beta \subset [n], |\al|=|\beta|=m+2, \\ \al\sim \beta}}   \E\big[ I_\al^{(m+2)}I_\beta^{(m+2)} \big]  &\sim \sum_{\ell=q+1}^{m+1} \frac{n^{2(\taual)-\tau_{\ell-1}}}{\ell! \big(( m+2-\ell)! \big)^2} \\
&\sim \frac{n^{2(\taual)-\tau_{q}}}{(q+1)! \big(( m-q+1)! \big)^2} \vee \frac{n^{2(\taual)-\tau_m}}{(m+1)!}. 
\end{split}
\end{align}
The key observation is that 
\begin{equation}  \label{e:taum.and.taum+1}
\tau_m-(\taual) = -1+\sum_{i=q}^m \binom{m+1}{i} \al_i > -1+\psi_m>0. 
\end{equation}
Combining \eqref{e:EXm}, \eqref{e:cross.moment}, and
\eqref{e:taum.and.taum+1} gives us \eqref{e:Lambda.Xm}, and 
\eqref{e:delta.Xm} follows from \eqref{e:EXm} and \eqref{e:Lambda.Xm}. 
\end{proof}

Before stating the next proposition, for $j\ge q+1$ we let $\mH_j$ be the collection of words of length $j$ in $[n]$. Let us fix an arbitrary
$i$-simplex for $i\in \{ q,\dots,m \}$, and let  $L$ be a simplicial
complex (of dimension $i$) that is induced by this
$i$-simplex.   
  Given $\beta\in \mH_{i+1}$,
define  
$$
J_\beta^{(i+1)}=\one \big\{  R_{i+1}(\beta)\subset \Gamma_\bp\big\}, 
$$
where  $R_{i+1}(\beta)$ denotes the  collection of words of length
$q+1,\dots,i+1$, that are contained in $\beta$; it is isomorphic to
$L\setminus L^{(q-1)}$, where $L^{(q-1)}$ is the $(q-1)$-skeleton of $L$.  In particular, $R_{q+1}(\beta) = \beta$ for
$\beta\in \mH_{q+1}$ (see \eqref{e:def.J.q+1}).  For a subset
$S\subset [n]$ define
$$
\mH_{i+1}(S) = \{ \beta\in \mH_{i+1}: V(\beta)\subset S \}, 
$$
where  $V(\beta)$ is the vertex set of $\beta$. We will be
particularly interested in $\mH_{i+1}(\U)$, where $\U$ is defined  in
\eqref{e:def.mU}. 
As an extension of \eqref{e:def.Yq}, define  $Y_i:= \sum_{\beta\in
  \mH_{i+1}(\U)}J_\beta^{(i+1)}$, which represents the number of
$i$-faces in $X(\U, \bp)$.  

The following proposition  is used for the proof of Proposition
\ref{p:lower.LD.leading.term2} 
  only in the special case $i=q$
(see \eqref{e:need.prop.i=q}). Nevertheless, we will prove it  under
a more general setup. The proof uses  ideas from Theorem 20 in \cite{janson:warnke:2016}. 

\begin{proposition}  \label{p:small.q.face}
Assume that \eqref{e:pos.assump} holds. 
Let $i\in\{q, \ldots, m\}$. Then for every $\vep\in (0,2^{-(m+5)})$, there exists $N\in \bbn$ such that for all $n\ge N$, 
$$
\P \big( X_m\le (1-\vep) \E[X_m] \big) \ge \frac{1}{6}\,\P \big( Y_i\le (1-2^{m+5}\vep)\E[Y_i] \big). 
$$
\end{proposition}
\begin{proof}
Set 
$$
Z:=\sum_{(\alpha, \beta)\in \mathcal H_{m+2}\times \mH_{i+1}}\hspace{-.5pt}\one \{ \beta \subset \alpha \}I_\al^{(m+2)} = \tau X_m, 
$$
where (independently of $\alpha$), $\tau:=\sum_{\beta\in
  \mH_{i+1}}\one\{ \beta\subset \alpha \}=\binom{m+2}{i+1}$. For $\beta\in \mH_{i+1}$ and $S\subset [n]$, define 
$$
\mathcal K(S,\beta) = \big\{ \alpha \in \mathcal H_{m+2}:\beta\subset
\alpha, \, V(\alpha)\setminus V(\beta)\subset [n]\setminus S  \big\}. 
$$
Define, for $S\subset [n]$, 
$$
Z_S:= \sum_{(\alpha, \beta):\, \beta\in \mH_{i+1}(S), \, \alpha\in K(S,\beta)}
  I_\alpha^{(m+2)}, 
$$
and $R_\U:=Z-Z_\U$. Additionally, let $z=(1-2^{m+4}\vep)\E[Z_\U]$ and $r=(1-\vep)\E[Z]-z$. By the FKG inequality, we have 
\begin{align*}
\P\big( X_m\le (1-\vep)\E[X_m] \big) &= \P(R_\U + Z_\U\le r+z) \\
&\ge \P(R_\U \le r, \, Z_\U\le z) \ge \P(R_\U\le r)\P(Z_\U\le z). 
\end{align*}
To complete the proof, it is sufficient to demonstrate that 
\begin{align}
\P(R_\U \le r) &\ge 1-\frac{\Lambda(X_m)}{\Lambda(X_m) +(\vep \mu(X_m))^2}, \label{e:suff.1st}\\
\P(Z_\U \le z) &\ge \frac{1}{2}\, \Big( 1-\frac{\Lambda(X_m)}{\Lambda (X_m)+2(\vep \mu(X_m))^2} \Big)\P\big(Y_i\le (1-2^{m+5}\vep)\E[Y_i]   \big), \label{e:suff.2nd}
\end{align}
Indeed, by 
\eqref{e:EXm}, Lemma \ref{l:Lambda.Delta} $(i)$ and \eqref{e:pos.assump}
we have  $(\vep \mu(X_m))^2 \ge \Lambda(X_m)$ for sufficiently large
$n$, so  \eqref{e:suff.1st} and \eqref{e:suff.2nd}  imply the claim of
Proposition \ref{p:small.q.face}. 

For the proof of \eqref{e:suff.1st}, let $\mathscr X$ be a random subset of $[n]$, whose elements are  chosen uniformly and independently of $\Gamma_\bp$, such that $|\mathscr X| = |\U|=\lfloor n/2 \rfloor$. Then, 
$$
\E[Z_{\mathscr X} | \Gamma_\bp] =\sum_{(\alpha, \beta)\in \mathcal H_{m+2}\times \mH_{i+1}}\one \{ \beta\subset \alpha \}\, \P \big( V(\alpha)\setminus V(\beta)\subset [n]\setminus \mathscr X, \, V(\beta)\subset \mathscr X \big)\, I_\alpha^{(m+2)}.  
$$
By an easy calculation, 
$$
\P \big( V(\alpha)\setminus V(\beta)\subset [n]\setminus \mathscr X, \, V(\beta)\subset \mathscr X \big) = \binom{n}{\lfloor n/2 \rfloor}^{-1}\binom{n-(m+2)}{\lfloor n/2 \rfloor-(i+1)}\to 2^{-(m+2)}, \ \ \ n \to\infty, 
$$
and so, for large $n$, 
$$
\E[Z_{\mathscr X}] = \E\big[ \E[Z_{\mathscr X} | \Gamma_\bp] \big] \ge 2^{-(m+3)} \E[Z]. 
$$
It follows from the independence of $\Gamma_\bp$ and $\mathscr X$,
together with the symmetry, that 
$$
\E[Z_{\mathscr X}] = \sum_{S\subset [n], \, |S|=|\U|} \E [Z_{\mathscr X} | \mathscr X=S] \, \P(\mathscr X=S) = \E[Z_\U], 
$$
and hence, for large $n$, 
\begin{equation}  \label{e:lower.bdd.E.ZU}
\E[Z_\U] \ge 2^{-(m+3)} \E[Z]. 
\end{equation}
For $\alpha, \alpha'\in \mathcal H_{m+2}$, the FKG inequality ensures that Cov$(I_\al^{(m+2)}, I_{\alpha'}^{(m+2)})\ge 0$, which in turn implies that 
$$
\text{Var}(R_\U) \le \text{Var}(Z) = \tau^2 \text{Var}(X_m) \le \tau^2 \Lambda(X_m). 
$$
From \eqref{e:lower.bdd.E.ZU} we see that 
$$
r-\E[R_\U] \ge \vep \E[Z] = \tau\vep \mu(X_m). 
$$
Finally, the one-sided Chebyshev's inequality, which is given in
\cite[Claim 16]{janson:warnke:2016}, with $\nu=\tau^2 \Lambda(X_m)$ and $t=\tau \vep \mu(X_m)$,  completes the proof of  \eqref{e:suff.1st}.

We now turn our attention to \eqref{e:suff.2nd}. We start by
introducing additional   notation: 
for $\al_1, \al_2 \in \mH_{m+2}$ and $\beta_1, \beta_2 \in \mH_{i+1}$
denote 
$$
K_{\al_1\setminus \beta_1}^{(m+2,i+1)} := \one \big\{ \text{all strict 
  subwords of $\alpha_1$ that are not subwords of $\beta_1$ are in} \ \Gamma_\bp   \big\}, 
$$
and
\begin{align*}
&K_{(\alpha_1\cup \alpha_2)\setminus (\beta_1\cup \beta_2)}^{(m+2,i+1)} := \one \Big\{ 
\text{all strict 
  subwords of $\alpha_1$ or  $\alpha_2$} \\
&\hskip 2in \text{
that are not subwords of either
  $\beta_1$ or  $\beta_2$ are in} \ \Gamma_\bp \Big\}. 
\end{align*}
We define two events, $\mathcal E$ and $\mathcal D$, as follows. 
Let $\mathcal E := \big\{  Y_i\le (1-2^{m+5}\vep)\E[Y_i] \big\}$. If
we denote 
$$
V_\U^+ := \sum_{(\beta_1, \beta_2)\in \mH_{i+1}(\U)\times
  \mH_{i+1}(\U)}  J_{\beta_1}^{(i+1)} J_{\beta_2}^{(i+1)}
\sum_{\substack{(\alpha_1,\alpha_2)\in \mathcal K(\U, \beta_1)\times
    \mathcal K(\U, \beta_2), \\  |\alpha_1\cap  \alpha_2|\geq q+1 }}
\E[K_{(\alpha_1\cup \alpha_2)\setminus (\beta_1\cup
  \beta_2)}^{(m+2,i+1)}], 
$$
then we set $\mathcal D:= \big\{  V_\U^+\le 2 \E[V_\U^+]
\big\}$. Appealing to the FKG inequality, as  well as Markov's inequality, 
$$
\P(\mathcal D\cap \mathcal E)\ge \P(\mathcal D)\P(\mathcal E)\ge \frac{1}{2}\, \P\big( Y_i \le (1-2^{m+5}\vep)\E[Y_i] \big). 
$$
Let $\Gamma_\bp^{(i+1)}(\U)$ be the subset of $\Gamma_\bp$ consisting
of words on $\U$ with  length at most $i+1$. 
Since $\mathcal D$ and $\mathcal E$ are both measurable with respect to $\Gamma_\bp^{(i+1)}(\U)$, 
\begin{equation}  \label{e:ind.outside}
\P(Z_\U\le z) \ge \E\Big[ \P \big( Z_\U\le z\, | \, \Gamma_\bp^{(i+1)}(\U)\big)\one_{\mathcal D\cap \mathcal E}  \Big]. 
\end{equation}
Consider the conditional probability 
$$
\P^*(\cdot) := \P \big( \cdot\, | \,  \Gamma_\bp^{(i+1)}(\U)\big), 
$$
and   the corresponding (conditional) mean  and variance, $\E^*$ and
$\text{Var}^*$. We claim that, on the event $\mathcal D\cap
\mathcal E$, 
\begin{align}
z-\E^*[Z_\U] &\ge 2 \tau \vep \mu (X_m), \label{e:exp.star} \\
\text{Var}^*(Z_\U)&\le 2 \tau^2 \Lambda(X_m). \label{e:var.star}
\end{align}
Once \eqref{e:exp.star} and \eqref{e:var.star} have been established,
the 
one-sided Chebyshev's inequality says that, on the event $\mathcal D\cap
\mathcal E$, 
$$
\P^*(Z_\U\ge z) \le \P^*\big( Z_\U -\E^*[Z_\U]\ge 2\tau \vep \mu(X_m) \big) \le \frac{\Lambda(X_m)}{\Lambda(X_m) + 2(\vep \mu(X_m))^2}. 
$$
Substituting this inequality into \eqref{e:ind.outside}, gives us \eqref{e:suff.2nd} as required. 

For the proof of \eqref{e:exp.star}, note that for $\beta\in \mH_{i+1}(\U)$ and $\alpha\in \mathcal K(\U, \beta)$, $J_\beta^{(i+1)}$ is measurable with respect to $\Gamma_\bp^{(i+1)}(\U)$, whereas $K_{\alpha \setminus \beta}^{(m+2,i+1)}$ is independent of $\Gamma_\bp^{(i+1)}(\U)$. Therefore, 
$$
\E^*[Z_\U] = \sum_{\beta\in \mH_{i+1}(\U)} J_\beta^{(i+1)} \E\Big[ \sum_{\alpha\in \mathcal K(\U, \beta)} K_{\alpha\setminus \beta}^{(m+2,i+1)}\Big]. 
$$
In the last expression, by symmetry, $W_\beta:=\sum_{\alpha\in
  \mathcal K(\U, \beta)}K_{\alpha \setminus \beta}^{(m+2,i+1)}$ has
expectation 
independent of the choice of $\beta\in \mH_{i+1}(\U)$. It  follows
that for any fixed $\tilde \beta\in \mH_{i+1}(\U)$, we have 
$\E^*[Z_\U]=Y_i\E[W_{\tilde \beta}]$, and, hence,
$\E[Z_\U]=\E[Y_i]\E[W_{\tilde \beta}]$.   Now on $\mathcal E$ we have $\E^*[Z_\U]\le (1-2^{m+5}\vep)\E[Z_\U]$ and by \eqref{e:lower.bdd.E.ZU}, 
$$
z-\E^*[Z_\U] \ge z-(1-2^{m+5}\vep)\E[Z_\U]\ge 2\tau \vep \mu(X_m), 
$$
proving that \eqref{e:exp.star} holds on $\mathcal E\supset \mathcal
D\cap \mathcal E$. 
Furthermore, arguments similar to those used for Equ.~(67) and (68) in
\cite{janson:warnke:2016}, show that  \eqref{e:var.star}
holds on $\mathcal D \supset \mathcal
D\cap \mathcal E$. 
\end{proof}
\medskip

Lemma \ref{l:variance.asym1} below calculates  variance asymptotics
for $S_j$,   $T_{j,r}$, and $V_{m+2}$, which are respectively
defined at \eqref{e:def.Sj},  \eqref{e:def.T.jr}, and \eqref{e:def.V.m+2}. Similarly, Lemma \ref{l:variance.asym2} gives variance asymptotics  for $S_{m+2}^\uparrow$, $S_j^{\uparrow}$, and $T_{m+2}^\downarrow$ (see \eqref{e:def.S.m+2.uparrow}, \eqref{e:def.Sj2.uparrow}, and \eqref{e:def.T.m+2.downarrow}, respectively). 
Since Lemmas \ref{l:variance.asym1} and  \ref{l:variance.asym2} can be
proven in the same way, we give only a proof of Lemma \ref{l:variance.asym1}. 

\begin{lemma}  \label{l:variance.asym1}
As $n\to\infty$, 
\begin{align}  \label{e:Var.S.m+2}
\text{Var}(S_{m+2}) \sim \text{Var}(V_{m+2}) \sim \begin{cases}
\frac{n^{2(\taual)-\tau_q}}{(q+1)! ((m-q+1)!)^2}  & \text{ if } \taual>\tau_q, \\[7pt]
\frac{n^\taual}{(m+2)!} & \text{ if } \taual <\tau_q, \\[7pt]
\Big\{ \frac{1}{(q+1)! ((m-q+1)!)^2}  + \frac{1}{(m+2)!}  \Big\}n^\taual & \text{ if } \taual =\tau_q.
\end{cases}
\end{align}
Further,
\begin{equation}  \label{e:Var.T.m+2}
\text{Var}(T_{m+2}) \le C^* \big(  n^{2(\taual)-\tau_q} \vee n^\taual\big). 
\end{equation}
For $j\ge m+3$, 
\begin{align}
&\text{Var}(S_j) \le C_j \big(  n^{2(\taual)-\tau_q} \vee n^\taual\big)  \label{e:Var.S.j}
\end{align}
for a $j$-dependent $C_j>0$.
For $j \ge m+3$ and $r\ge 1$, 
\begin{equation}  \label{e:Var.T.jr}
\text{Var}(T_{j,r} ) \le C_j \big(  n^{2(\taual)-\tau_q} \vee n^\taual\big)\, n^{1-\psi_m}. 
\end{equation}
Finally, 
\begin{equation}  \label{e:Var.V.m+2-T.m+2}
\text{Var}(V_{m+2}-T_{m+2}) \le C^* \big(  n^{2(\taual)-\tau_q} \vee n^\taual\big)\, n^{(1\vee \psi_{m-1})-\psi_m}. 
\end{equation}
\end{lemma}

\begin{lemma}  \label{l:variance.asym2}
For $j\ge m+2$, 
\begin{align}
\text{Var}(S_{j}^\uparrow) &\le C_j\big( v_{N+1}^{2(\taual)-\tau_q} \vee v_{N+1}^\taual\big), \label{e:Var.S.m+2.uparrow}
\end{align}
and
\begin{align}
\text{Var}(T_{m+2}^\downarrow) &\le C^*\big( v_{N}^{2(\taual)-\tau_q} \vee v_{N}^\taual\big).   \label{e:Var.T.m+2.downarrow}
\end{align}
\end{lemma}

\begin{proof}[Proof of Lemma \ref{l:variance.asym1}] We begin with the proof of \eqref{e:Var.S.m+2}. \\
\textit{Proof of \eqref{e:Var.S.m+2}}: For ease of notation, we drop the superscript from $\xi_\sigma^{(m+2,m)}$ and write that 
\begin{align}
\begin{split}  \label{e:decomp.var.S.m+2}
\text{Var}(S_{m+2})  
&=\sum_{\substack{\sigma \subset [n], \\ |\sigma|=m+2}}\sum_{\substack{\tau \subset [n], \\ |\tau|=m+2}}  \big\{ \E[\xi_\sigma \xi_\tau] - \E[\xi_\sigma]\E[\xi_\tau] \big\}    \\
&=\sum_{\ell=0}^{m+2} \sum_{\substack{\sigma \subset [n], \\ |\sigma|=m+2}}\sum_{\substack{\tau \subset [n], \\ |\tau|=m+2}}  \big\{ \E[\xi_\sigma \xi_\tau] - \E[\xi_\sigma]\E[\xi_\tau] \big\} \, \one \{ |\sigma\cap \tau|=\ell \}  \\
&= \sum_{\ell=0}^{m+2} \binom{n}{m+2}\binom{m+2}{\ell} \binom{n-m-2}{m+2-\ell} \\
&\hskip 1in \big\{ \E[\xi_{\sigma_{m+2}} \xi_{\tau_{m+2}}] - \E[\xi_{\sigma_{m+2}}]\E[\xi_{\tau_{m+2}}] \big\} \, \one \{ |\sigma_{m+2}\cap \tau_{m+2}|=\ell \}. 
\end{split}
\end{align}
For every $\sigma\subset [n]$ with $|\sigma|=m+2$, 
\begin{equation}  \label{e:1st.xi.sigma}
\E[\xi_\sigma] = \prod_{i=q}^m p_i^{\binom{m+2}{i+1}} = n^{-\sum_{i=q}^m \binom{m+2}{i+1}\al_i}. 
\end{equation}
Let $\sigma,\tau\subset [n]$ with $|\sigma|=|\tau|=m+2$. 
If $\ell=|\sigma\cap \tau|\in \{ 0,\dots,q \}$, 
then $\xi_\sigma$ and $\xi_\tau$ are independent, because all the
$(\ell-1)$-faces exist in $X([n],\bp)$ with probability
$1$. Therefore, 
$\E[\xi_\sigma\xi_\tau] = \E[\xi_\sigma]\E[\xi_\tau]$. 
Next, if  $\ell\in \{ q+1,\dots,m+1 \}$, then $\sigma \cap \tau$ contains $\binom{\ell}{i+1}$ faces of dimension $i$ for each $q\le i \le \ell-1$; hence, 
\begin{equation}  \label{e:2nd.xi.sigma.xi.tau}
\E[\xi_\sigma\xi_\tau]=\prod_{i=q}^m p_i^{\binom{m+2}{i+1}} \times \prod_{i=q}^m p_i^{\binom{m+2}{i+1}-\binom{\ell}{i+1}} = n^{-2\sum_{i=q}^m \binom{m+2}{i+1}\al_i +\sum_{i=q}^{\ell-1}\binom{\ell}{i+1}\al_i}. 
\end{equation}
In this case, by \eqref{e:1st.xi.sigma} and \eqref{e:2nd.xi.sigma.xi.tau}, as $n\to\infty$, 
\begin{align*}
&\binom{n}{m+2}\binom{m+2}{\ell} \binom{n-m-2}{m+2-\ell} \big\{ \E[\xi_\sigma \xi_\tau] - \E[\xi_\sigma]\E[\xi_\tau] \big\} \\ 
&\sim \frac{n^{2(m+2)-\ell}}{\ell! \big( (m+2-\ell)! \big)^2}\,\Big( n^{-2\sum_{i=q}^m \binom{m+2}{i+1}\al_i +\sum_{i=q}^{\ell-1}\binom{\ell}{i+1}\al_i} - n^{-2\sum_{i=q}^m \binom{m+2}{i+1}\al_i }  \Big)\\
&\sim \frac{n^{2(\taual)-\tau_{\ell-1}}}{\ell! \big( (m+2-\ell)! \big)^2}.
\end{align*}
Finally, when $\ell=m+2$, we have as $n\to\infty$, 
\begin{align*}
&\binom{n}{m+2}\binom{m+2}{m+2} \binom{n-m-2}{0} \big\{ \E[\xi_\sigma \xi_\tau] - \E[\xi_\sigma]\E[\xi_\tau] \big\} \\ 
&\sim \frac{n^{m+2}}{(m+2)!}\, \big\{ \E[\xi_\sigma]- \big( \E[\xi_\sigma] \big)^2 \big\}\sim \frac{n^\taual}{(m+2)!}. 
\end{align*}
Substituting these asymptotic results back into \eqref{e:decomp.var.S.m+2}, 
\begin{align*}
\text{Var}(S_{m+2}) &\sim \sum_{\ell=q+1}^{m+1}\frac{n^{2(\taual)-\tau_{\ell-1}}}{\ell! \big( (m+2-\ell)! \big)^2} + \frac{n^\taual}{(m+2)!}  \\
&\sim \left( \frac{n^{2(\taual)-\tau_q}}{(q+1)!\big( (m-q+1)! \big)^2} \vee \frac{n^{2(\taual)-\tau_m}}{(m+1)!}\right) + \frac{n^\taual}{(m+2)!}, \ \ \ n\to\infty.
\end{align*}
Since $\tau_m>\taual$ (see \eqref{e:taum.and.taum+1}),  this is the
asymptotics claimed for Var$(S_{m+2})$ in \eqref{e:Var.S.m+2}. 
One can derive variance asymptotics for $V_{m+2}$ in a nearly
identical  manner. 
\medskip

\noindent \textit{Proof of \eqref{e:Var.S.j}}: Removing again the
superscript from $\xi_\sigma^{(j,m)}$, we have  
$$
\text{Var}(S_j) = \sum_{\ell=0}^j \binom{n}{j} \binom{j}{\ell} \binom{n-j}{j-\ell} \big\{ \E[\xi_{\sigma_j}\xi_{\tau_j}] - \E[\xi_{\sigma_j}]\E[\xi_{\tau_j}] \big\}\, \one \{ |\sigma_j \cap \tau_j|=\ell \}. 
$$
Let $\sigma,\tau\subset [n]$ with $|\sigma|=|\tau|=j$. 
 As before, if $\ell=|\sigma\cap \tau|\in\{ 0,\dots,q \}$, then  $\E[\xi_\sigma\xi_\tau] = \E[\xi_\sigma]\E[\xi_\tau]$. 
Next, if $\ell\in \{q+1,\dots,m+1\}$, then
\begin{align}
\begin{split}  \label{e:decomp.var.S.j.l.in.0.q}
\binom{n}{j} \binom{j}{\ell} \binom{n-j}{j-\ell} \big\{ \E[\xi_\sigma\xi_\tau] - \E[\xi_\sigma]\E[\xi_\tau] \big\} &\le n^{2j-\ell} \E[\xi_\sigma \xi_\tau] \\
&= n^{2j-\ell} \sum_{K: |K|=j} \sum_{K': |K'|=j} \P \big(A_\ell(K,K')\big), 
\end{split}
\end{align}
where the sum is taken over all $m$-strongly connected complexes on
$j$ vertices, with at least one trivial or non-trivial $m$-cycle, up
to an isomorphism class, and 
\begin{align}  \label{e:def.AKK'}
A_\ell(K,K') := \big\{ &X(\sigma,\bp)  \  \text{contains a
  subcomplex isomorphic to $K$,} \\
\notag   &X(\tau,\bp)  \  \text{contains a
  subcomplex isomorphic to $K'$}\big\}. 
\end{align}
 In order to bound the sum in \eqref{e:decomp.var.S.j.l.in.0.q},
 observe that $\P\big(A_\ell(K,K')\big)$ is maximized when all of the
 $\ell$ common points of $\sigma$ and $\tau$ 
are taken from a single $(\ell-1)$-face. Then, according to Lemma \ref{l:estimate.st.conn}, 
\begin{align*}
\P\big(A_\ell(K,K') \big) &\le C_j \Big\{ \prod_{i=q}^m p_i^{\binom{m+2}{i+1}} \Big( \prod_{i=q}^m p_i^{\binom{m}{i}} \Big)^{j-m-2} \Big\}^2\prod_{i=q}^{\ell-1}p_i^{-\binom{\ell}{i+1}} \\
&=C_j n^{-2\sum_{i=q}^m \binom{m+2}{i+1}\al_i +\sum_{i=q}^{\ell-1}\binom{\ell}{i+1}\al_i} (n^{-\psi_m})^{2(j-m-2)}. 
\end{align*}
Here, the factor $\prod_{i=q}^{\ell-1}p_i^{-\binom{\ell}{i+1}}$ is
added due to the  double count of the common
$(\ell-1)$-face. Substituting this bound in
\eqref{e:decomp.var.S.j.l.in.0.q}, 
$$
n^{2j-\ell} \sum_{K: |K|=j} \sum_{K': |K'|=j} \P \big(A_\ell(K,K')\big) \le C_j n^{2(\taual)-\tau_{\ell-1}} (n^{1-\psi_m})^{2(j-m-2)}. 
$$

We next consider the case $\ell\in \{ m+2,\dots,j-1 \}$.  Once again,
we bound the final expression in \eqref{e:decomp.var.S.j.l.in.0.q}. 
In this case, the maximum of $\P \big(A_\ell(K,K')\big)$ is attained
when $\sigma$ and $\tau$ have $\ell-m$ faces of dimension $m$ in
common, such that these $m$-faces are strongly connected, in a way
that each of the $m$-faces shares $m$ points with at least one of the
remaining  $m$-faces. Taking this  into consideration, we  conclude that 
\begin{align}  
\begin{split} \label{e:many.overlapping}
\P \big(A_\ell(K,K')\big) &\le C_j \Big\{ \prod_{i=q}^m p_i^{\binom{m+2}{i+1}} \Big( \prod_{i=q}^m p_i^{\binom{m}{i}} \Big)^{j-m-2} \Big\}^2 \prod_{i=q}^m p_i^{-\binom{m+1}{i+1}} \Big( \prod_{i=q}^m p_i^{-\binom{m}{i}} \Big)^{\ell-m-1} \\
&= C_j n^{-2\sum_{i=q}^m \binom{m+2}{i+1}\al_i + \sum_{i=q}^m \binom{m+1}{i+1}\al_i } (n^{-\psi_m})^{2j-m-\ell-3}. 
\end{split}
\end{align}
We refer the reader to  the proof of \cite[Theorem
4.2]{owada:samorodnitsky:thoppe:2021} for  a  more detailed derivation
of  \eqref{e:many.overlapping}. Using this estimate,  we see  that 
$$
n^{2j-\ell} \sum_{K: |K|=j} \sum_{K': |K'|=j} \P \big(A_\ell(K,K')\big)\le C_j n^{2(\taual)-\tau_m} (n^{1-\psi_m})^{2j-m-\ell-3}. 
$$

Finally, for  $\ell=j$, a direct application of Lemma
\ref{l:estimate.st.conn} gives us 
\begin{align*}
&\binom{n}{j} \binom{j}{j} \binom{n-j}{0} \big\{ \E[\xi_\sigma] - \big(\E[\xi_\sigma]\big)^2 \big\}  \\ 
\le& n^j \sum_{K: |K|=j} \P\big( X(\sigma,\bp)  \ \text{contains
     a subcomplex isomorphic to}\ K \big)\\
\le& C_j n^j \prod_{i=q}^m p_i^{\binom{m+2}{i+1}} \Big( \prod_{i=q}^m p_i^{\binom{m}{i}} \Big)^{j-m-2} 
=C_j n^\taual (n^{1-\psi_m})^{j-m-2}. 
\end{align*}
Putting all of these results together, we obtain 
\begin{align*}
\text{Var}(S_j) &\le C_j \Big\{  \bigvee_{\ell=q+1}^{m+1} n^{2(\taual)-\tau_{\ell-1}} (n^{1-\psi_m})^{2(j-m-2)} \\
&\qquad \qquad \vee \bigvee_{\ell=m+2}^{j-1} n^{2(\taual)-\tau_m} (n^{1-\psi_m})^{2j-m-\ell-3}  \vee n^{\taual}(n^{1-\psi_m})^{j-m-2} \Big\}  \\
&\le C_j \big\{ n^{2(\taual)-\tau_q} \vee n^\taual \big\}. 
\end{align*}
The last inequality follows from \eqref{e:taum.and.taum+1} and $\psi_m>1$. 
\medskip


\noindent \textit{Proof of \eqref{e:Var.T.m+2} and \eqref{e:Var.T.jr}}: For $j\ge m+2$ and $r\ge1$,  dropping the superscript from $\eta_\sigma^{(j,r,m)}$ as before,  we have that 
$$
\text{Var}(T_{j,r}) = \sum_{\ell=0}^j \binom{n}{j} \binom{j}{\ell} \binom{n-j}{j-\ell} \big\{ \E[\eta_{\sigma_j}\eta_{\tau_j}] - \E[\eta_{\sigma_j}]\E[\eta_{\tau_j}] \big\}\, \one \{ |\sigma_j \cap \tau_j|=\ell \} =: \sum_{\ell=0}^j \mT_{\ell,j}. 
$$ 
The following estimates were established in  the proof of \cite[Theorem
4.2]{owada:samorodnitsky:thoppe:2021}: 
\begin{equation}  \label{e:mT.ell}
\mT_{\ell,j} = \begin{cases}
0 & \text{ if } \ell\in\{ 0,\dots,q-2 \}, \\[3pt]
\mathcal O\big( n^{2(\taual)-\tau_q} (n^{1-\psi_m})^{2(j-m)-2} \big) & \text{ if } \ell=q-1, \\[3pt]
\mathcal O\big( n^{2(\taual)-\tau_q} (n^{1-\psi_m})^{2(j-m)-3} \big) & \text{ if } \ell=q, \\[3pt]
\mathcal O\big( n^{2(\taual)-\tau_{\ell-1}} (n^{1-\psi_m})^{2(j-m)-4} \big) & \text{ if } \ell\in \{  q+1,\dots,m+1\}, \\[3pt]
\mathcal O\big( n^{2(\taual)-\tau_m} (n^{1-\psi_m})^{2j-m-\ell-3} \big) & \text{ if } \ell\in \{ m+2,\dots,j-1\}, \\[3pt]
\mathcal O\big( n^{\taual} (n^{1-\psi_m})^{j-m-2} \big) & \text{ if } \ell=j, \\[3pt]
\end{cases}
\end{equation}
where $\beta_n = \mathcal O(\gamma_n)$ means that there exists a $j$-dependent $C_j>0$ such that
$\beta_n/\gamma_n \le C_j$ for all $n\ge1$. 
Recalling \eqref{e:taum.and.taum+1}, it follows from 
\eqref{e:mT.ell}  that $\text{Var}(T_{j,r}) \le C_j (a_n \vee b_n)$, 
where 
\begin{align*}
a_n &:= \bigvee_{\ell=q+1}^{m+1} n^{2(\taual)-\tau_{\ell-1}} (n^{1-\psi_m})^{2(j-m)-4}, \\
b_n &:= n^\taual (n^{1-\psi_m})^{j-m-2}. 
\end{align*}
If $j=m+2$, then   by \eqref{e:taum.and.taum+1}, $a_n \vee b_n =
n^{2(\taual)-\tau_q} \vee n^\taual$, proving \eqref{e:Var.T.m+2}.
If $j\ge m+3$ and $r\ge1$, then 
\begin{align*}
a_n &\le \big\{  n^{2(\taual)-\tau_q}  \vee n^{2(\taual)-\tau_m} \big\}\, n^{1-\psi_m}, \\
 b_n &\le n^{\taual+1-\psi_m}, 
\end{align*}
so that 
$$
a_n   \vee b_n \le \big\{ n^{2(\taual)-\tau_q} \vee n^\taual \big\}\,
n^{1-\psi_m}, 
$$
and \eqref{e:Var.T.jr} follows. 
\medskip

\noindent \textit{Proof of \eqref{e:Var.V.m+2-T.m+2}}: We use the same 
style of  notation as in \eqref{e:decomp.var.S.m+2}. We have 
\begin{align*}
\text{Var}(V_{m+2}-T_{m+2}) &=\sum_{\ell=0}^{m+2}\binom{n}{m+2}\binom{m+2}{\ell} \binom{n-m-2}{m+2-\ell}  \\
&\times \Big\{  \E\big[ (\lambda_{\sigma_{m+2}}
                                                                                                               -\eta_{\sigma_{m+2}})(\lambda_{\tau_{m+2}}-\eta_{\tau_{m+2}}) \big] - \E[\lambda_{\sigma_{m+2}}-\eta_{\sigma_{m+2}}]\E[\lambda_{\tau_{m+2}}-\eta_{\tau_{m+2}}] \Big\}  \\
  &\hskip 3.5in \times \one \{ |\sigma_{m+2} \cap \tau_{m+2}|=\ell \} \\
&=\sum_{\ell=0}^{m+2}\binom{n}{m+2}\binom{m+2}{\ell}
  \binom{n-m-2}{m+2-\ell}          (A_n-2B_n +C_n) \one \{ |\sigma_{m+2} \cap \tau_{m+2}|=\ell \} 
\\
&:= \sum_{\ell=0}^{m+2} \U_\ell, 
\end{align*}
where 
\begin{align*}
A_n& :=  \E[\lambda_{\sigma_{m+2}}\lambda_{\tau_{m+2}}]
     -\E[\lambda_{\sigma_{m+2}}]\E[\lambda_{\tau_{m+2}}],  \ \ \ B_n :=
      \E[\lambda_{\sigma_{m+2}}\eta_{\tau_{m+2}}]
     -\E[\lambda_{\sigma_{m+2}}]\E[\eta_{\tau_{m+2}}], \\
      C_n& := \E[\eta_{\sigma_{m+2}}\eta_{\tau_{m+2}}]
     -\E[\eta_{\sigma_{m+2}}]\E[\eta_{\tau_{m+2}}]. 
\end{align*}
Let $K$ be the  non-trivial $m$-cycle on $m+2$ points. For $\sigma \subset [n]$ with $|\sigma |=m+2$, define 
$$
A(K) := \big\{  X(\sigma,\bp)  \text{ is isomorphic to } K  \big\}, 
$$
so that 
\begin{equation}  \label{e:def.qK}
q_K:= \P(A(K)) = \prod_{i=q}^m p_i^{\binom{m+2}{i+1}}(1-p_{m+1}). 
\end{equation}
Let $\sigma,\tau \subset [n]$ with $|\sigma |=|\tau|=m+2$. 
We break our discussion into four different cases, depending on the
value of $\ell=|\sigma \cap \tau|$. We continue using the notation
\eqref{e:def.AKK'}.  
\medskip

\noindent \textit{\underline{Case I}}: $\ell \in \{0,\dots, q-1\}$. \\
Since all faces of dimension smaller than $q$ exist in $X([n],\bp)$ with probability
$1$, we have $\E[\lambda_\sigma \lambda_\tau] = q_K^2 =
\E[\lambda_\sigma]\E[\lambda_\tau]$,
so  that $A_n=0$. Moreover, 
$$
\E[\lambda_\sigma\eta_\tau] = \P\big( A_\ell(K,K) \big)\P\big(\tau \text{ is maximal } | \,A_\ell(K,K)\big). 
$$
Here, it is clear  that $\P\big( A_\ell(K,K) \big)=q_K^2$.
Furthermore, the event $\{\tau \text{ is maximal}\}$ is, clearly,
independent of the restriction of the complex to $\sigma$. It thus follows 
that $B_n=0$.   We conclude by \eqref{e:mT.ell} with $j=m+2$ that 
\begin{equation}  \label{e:U.ell.0.q-1}
\U_\ell = \mT_{\ell,m+2} = \begin{cases}
0 & \text{ if } \ell\in\{ 0,\dots,q-2 \}, \\[3pt]
\mathcal O \big( n^{2(\taual)-\tau_q +2(1-\psi_m)} \big) & \text{ if }
\ell=q-1. 
\end{cases}
\end{equation}

\medskip

\noindent \textit{\underline{Case II}}: $\ell =q$. \\
Since all the $(q-1)$-faces exist with probability $1$, we still have
$\E[\lambda_\sigma \lambda_\tau] =
\E[\lambda_\sigma]\E[\lambda_\tau]$, and, hence,  $A_n=0$.  Next,   
$$
\E[\lambda_\sigma \eta_\tau] = q_K^2 \P \big( \tau \text{ is maximal } |\, A_q(K,K) \big) \le q_K^2 \P\Big( \bigcap_{v\in (\sigma \cap \tau)^c} B_v^c  \,\Big|\, A_q(K,K) \Big), 
$$
where  $B_v$ is the event  that $v$ forms an $m$-face with one of the
$(m-1)$-faces in $\tau$. Using independence and the argument   in
\eqref{e:E.T.m+2.downarrow}, we have 
\begin{align*}
\P\Big( \bigcap_{v\in (\sigma \cap \tau)^c} B_v^c \, \Big|\, A_q(K,K) \Big) &= \prod_{v\in (\sigma \cap \tau)^c} \Big( 1-\P\big(B_v\, |\, A_q(K,K)  \big) \Big) \\
&=\Big( 1-\binom{m+2}{m}n^{-\psi_m} +u_n  \Big)^{n-2(m+2)+q}, 
\end{align*}
where $u_n = o(n^{-\psi_m})$. Therefore, 
\begin{align*}
B_n&\le q_K^2 \Big( 1-\binom{m+2}{m}n^{-\psi_m} +u_n  \Big)^{n-2(m+2)+q} \hspace{-10pt}- q_K^2 \Big( 1-\binom{m+2}{m}n^{-\psi_m} +u_n  \Big)^{n-m-2}  \le C^* q_K^2 n^{-\psi_m}, 
\end{align*}
and, by \eqref{e:def.qK},
\begin{align}
\begin{split}  \label{e:Cn.bound}
\binom{n}{m+2}\binom{m+2}{q}\binom{n-m-2}{m+2-q} B_n &\le C^* n^{2(m+2)-q} q_K^2 n^{-\psi_m}  \le C^* n^{2(\taual)-\tau_q+1-\psi_m}. 
\end{split}
\end{align}
Finally, by \eqref{e:mT.ell}, 
\begin{equation}  \label{e:An.bound}
\binom{n}{m+2}\binom{m+2}{q}\binom{n-m-2}{m+2-q} C_n  =\mT_{q,m+2} = \mathcal O \big( n^{2(\taual)-\tau_q+1-\psi_m} \big). 
\end{equation}
Combining \eqref{e:Cn.bound} and \eqref{e:An.bound} gives us 
\begin{equation}  \label{e:U.ell.q}
\U_q \le C^* n^{2(\taual)-\tau_q+1-\psi_m}. 
\end{equation}

\noindent \textit{\underline{Case III}}: $\ell\in \{ q+1,\dots,m+1 \}$. \\
It follows from $\lambda_\sigma \ge \eta_\sigma$ and $\lambda_\tau \ge \eta_\tau$ that 
\begin{align*}
\U_\ell &\le n^{2(m+2)-\ell} \E \big[ (\lambda_{\tau} -\eta_{\tau})\lambda_{\sigma} \big] \\ 
&=n^{2(m+2)-\ell}\P\big( A_\ell(K,K) \big)\P\big( \tau \text{ is not maximal }\big|\, A_\ell(K,K) \big). 
\end{align*}
Since there are $\ell$ common points between $\sigma$ and $\tau$, 
\begin{align*}
\P\big( A_\ell(K,K) \big) &=\prod_{i=q}^m p_i^{\binom{m+2}{i+1}} (1-p_{m+1}) \times \prod_{i=q}^m p_i^{\binom{m+2}{i+1}-\binom{\ell}{i+1}} (1-p_{m+1}) \le \prod_{i=q}^m p_i^{2\binom{m+2}{i+1}-\binom{\ell}{i+1}}.
\end{align*}
Suppose first that $\ell\in \{q+1,\dots,m-1\}$. 
By the union bound, the probability that $\tau$ is not  maximal,
conditioned on $A_\ell(K,K)$, is bounded by $\binom{m+2}{m}$ times the
conditional probability that there exists at least one $m$-face
between a fixed $(m-1)$-face in $\tau$ and any  points outside of
$\tau$. The conditional probability of forming an $m$-face from an
$(m-1)$-face in $\tau$ and a point in $[n]\setminus (\tau\cup \sigma)$
is given by $\prod_{i=q}^m p_i^{\binom{m}{i}}=n^{-\psi_m}$, and there are 
$n-2(m+2)+\ell$ such points. On the other hand, if the candidate
point is in $\sigma \setminus \tau$, it suffices to consider the $(m-1)$-face in $\tau$ containing all $\ell$ common points
between $\sigma$ and $\tau$. Then the
above conditional probability of forming an $m$-face is
$$
\prod_{i=q}^m p_i^{\binom{m}{i}} \left( \prod_{i=q}^\ell p_i^{\binom{\ell}{i}}\right)^{-1}
= n^{\psi_\ell-\psi_m}, 
$$
which  is further bounded by $n^{(1\vee \psi_{m-1})-\psi_m}$.

On the other hand, if $\ell\in\{m,m+1\}$, then 
  $ \P\big( \tau \text{ is not maximal }\big|\, A_\ell(K,K) \big) =1$.
We thus  conclude that 
\begin{equation}  \label{e:U.ell.q+1.m+1}
\U_\ell \le C^* n^{2(m+2)-\ell} \prod_{i=q}^m
p_i^{2\binom{m+2}{i+1}-\binom{\ell}{i+1}} n^{(1\vee\psi_{m-1})-\psi_m}
= C^* n^{2(\taual)-\tau_{\ell-1}+(1\vee\psi_{m-1})-\psi_m}
\end{equation}
for $\ell\in\{q+1,\ldots, m-1\}$ and
\begin{equation}  \label{e:U.ell.m.m+1}
  \U_\ell \leq C^* n^{2(\taual)-\tau_{\ell-1}} \le C^* n^{2(\taual)-\tau_{m}}
  \end{equation}
for $\ell\in\{m,m+1\}$. 

\noindent \textit{\underline{Case IV}}: $\ell=m+2$. \\
In this case $\sigma=\tau$ and, as above, 
\begin{align}
\begin{split}  \label{e:U.ell.m+2}
\U_{m+2} &\le n^{m+2} \E[\lambda_\tau-\eta_\tau] =n^{m+2} q_K \P ( \tau \text{ is not maximal } |\, A(K) ) \\
&\le C^* n^{m+2} \prod_{i=q}^m p_i^{\binom{m+2}{i+1}} n^{1-\psi_m} =C^* n^{\taual +1 -\psi_m}. 
\end{split}
\end{align}
\medskip
Combining all the results in \eqref{e:U.ell.0.q-1}, \eqref{e:U.ell.q},
\eqref{e:U.ell.q+1.m+1}, \eqref{e:U.ell.m.m+1} 
and \eqref{e:U.ell.m+2}, we  conclude  that   
\begin{align*}
&\text{Var}(V_{m+2}-T_{m+2}) \le C^* \Big(  n^{2(\taual)-\tau_q+2(1-\psi_m)} \vee n^{2(\taual)-\tau_q+1-\psi_m} \\
&\hskip 1in \vee \bigvee_{\ell=q+1}^{m-1} n^{2(\taual)-\tau_{\ell-1}+(1\vee \psi_{m-1})-\psi_m} 
\vee n^{2(\taual)-\tau_m}
 \vee n^{\taual +1-\psi_m} \Big) \\
&\hskip 1in \le C^* \big( n^{2(\taual)-\tau_q} \vee n^\taual \big) \,
n^{(1\vee \psi_{m-1})-\psi_m}.  
\end{align*}
\end{proof}


\end{document}